\documentclass[10pt,reqno]{amsart}
\usepackage{amssymb}
\usepackage{amsmath}
\usepackage{amsthm}
\usepackage{graphicx}

\usepackage{subfigure}

\usepackage{cite}

\usepackage[colorlinks=true,urlcolor=blue,
citecolor=red,linkcolor=blue,linktocpage,pdfpagelabels,
bookmarksnumbered,bookmarksopen]{hyperref}%per colore-link

\usepackage{ifthen} % pacchetto necessario per un minimo di
\provideboolean{shownotes} % dichiariamo una variabile booleana
\setboolean{shownotes}{true} % la definiamo come "vera" o "falsa"
\newcommand{\margnote}[1]{
\ifthenelse{\boolean{shownotes}}%
{\marginpar{\raggedright\tiny\texttt{#1}}}%
{}%
}

\newcommand{\hole}[1]{
\ifthenelse{\boolean{shownotes}}%
{\begin{center} \fbox{ \rule {.25cm}{0cm}
\rule[-.1cm]{0cm}{.4cm} \parbox{.85\textwidth}{\begin{center}
\texttt{#1}\end{center}} \rule {.25cm}{0cm}}\end{center}}
{}
}

\usepackage[titletoc,title]{appendix}

\numberwithin{equation}{section} %riparte da zero ogni sezione

\newcounter{cont}[section]

\newtheorem{theorem}[cont]{Theorem}

\newtheorem{corollary}[cont]{Corollary}
\newtheorem{lemma}[cont]{Lemma}

\theoremstyle{definition}
\newtheorem{remark}[cont]{Remark}
\newtheorem{definition}[cont]{Definition}

 \theoremstyle{remark}

\renewcommand{\Re}{\mathrm{Re}\,}

\newcommand{\e}{\varepsilon}

\newcommand{\brm}{\overline{m}}

\newcommand{\rhos}{\rho_*}
\newcommand{\Us}{U_*}
\newcommand{\ms}{m_*}

\newcommand{\N}{\mathbb{N}}
\newcommand{\R}{\mathbb{R}}
\newcommand{\C}{\mathbb{C}}

\newcommand{\Z}{\mathbb{Z}}

\newcommand{\hU}{\widehat{U}}
\newcommand{\hV}{\widehat{V}}
\newcommand{\hA}{\widehat{A}}
\newcommand{\hB}{\widehat{B}}
\newcommand{\hK}{\widehat{K}}

\newcommand{\brc}{\langle \xi \rangle}
\newcommand{\La}{\Lambda}

\newcommand{\tiA}{\widetilde{A}}
\newcommand{\tiB}{\widetilde{B}}

\newcommand{\cF}{{\mathcal{F}}}
\newcommand{\cE}{{\mathcal{E}}}
\newcommand{\cL}{{\mathcal{L}}}

\newcommand{\cA}{{\mathcal{A}}}
\newcommand{\cB}{{\mathcal{B}}}
\newcommand{\cS}{{\mathcal{S}}}
\newcommand{\cT}{{\mathcal{T}}}
\newcommand{\cM}{{\mathcal{M}}}

\newcommand{\balp}{\overline{\alpha}}
\newcommand{\bbet}{\overline{\beta}}
\newcommand{\bA}{\overline{A}}

\newcommand{\vertiii}[1]{{\left\vert\kern-0.25ex\left\vert\kern-0.25ex\left\vert #1 
    \right\vert\kern-0.25ex\right\vert\kern-0.25ex\right\vert}}

\newcommand{\pd}{\partial}
\newcommand{\bw}{\overline{w}}
\newcommand{\br}{\overline{\rho}}
\newcommand{\mc}{m_*}
\newcommand{\rc}{\rho_*}
\newcommand{\tL}{\widetilde{\cL}}
\newcommand{\hw}{\hat{w}}
\newcommand{\tw}{\tilde{w}}

\begin{document}

\title[Well-posedness and decay structure of a QHD system with viscosity]{Well-posedness and decay structure of a quantum hydrodynamics system with Bohm potential and linear viscosity}

\author[R. G. Plaza]{Ram\'on G. Plaza}

\address[R. G. Plaza]{Departamento de Matem\'aticas y Mec\'anica\\Instituto de 
Investigaciones en Matem\'aticas Aplicadas y en Sistemas\\Universidad Nacional Aut\'onoma de 
M\'exico\\Circuito Escolar s/n, Ciudad Universitaria C.P. 04510 Cd. Mx. (Mexico)}

\email{plaza@mym.iimas.unam.mx}

\author[D. Zhelyazov]{Delyan Zhelyazov}

\address[D. Zhelyazov]{Departamento de Matem\'aticas y Mec\'anica\\Instituto de 
Investigaciones en Matem\'aticas Aplicadas y en Sistemas\\Universidad Nacional Aut\'onoma de 
M\'exico\\Circuito Escolar s/n, Ciudad Universitaria C.P. 04510 Cd. Mx. (Mexico)}

\email{delyan.zhelyazov@iimas.unam.mx}



%%%%%%%%%%%%%%%%%%%%%%%%%%%%%%%\nocite{*}

\begin{abstract} 
In this paper, a compressible viscous-dispersive Euler system in one space dimension in the context of quantum hydrodynamics is considered. The purpose of this study is twofold. First, it is shown that the system is locally well-posed. For that purpose, the existence of classical solutions which are perturbation of constant states is established. Second, it is proved that in the particular case of subsonic equilibrium states, sufficiently small perturbations decay globally in time. In order to prove this stability property, the linearized system around the subsonic state is examined. Using an appropriately constructed compensating matrix symbol in the Fourier space, it is proved that solutions to the linear system decay globally in time, underlying a dissipative mechanism of regularity gain type. These linear decay estimates, together with the local existence result, imply the global existence and the decay of perturbations to constant subsonic equilibrium states as solutions to the full nonlinear system.
\end{abstract}

%%%%%%%%%%%%%%%%%%%%%%%%%%%\tableofcontents

\maketitle

\section{Introduction}
\label{sec:intro}
Consider the following quantum hydrodynamics (QHD) system with linear viscosity,
\begin{equation}
\label{QHD}
	\begin{cases}
		\rho_t + m_x=0,\\
		m_t +\left(\displaystyle\frac{m^2}{\rho}+p(\rho)\right)_x=\mu m_{xx}+k^2\rho\left(\displaystyle\frac{(\sqrt{\rho})_{xx}}{\sqrt{\rho}}\right)_x,
	\end{cases}
\end{equation}
where $\rho > 0$ is the density, $m=\rho u$ denotes the momentum ($u$ is the velocity), $p(\rho) = \rho^\gamma$ with $\gamma > 1$ is the pressure function and $\mu > 0$, and $k > 0$ are positive constants representing viscosity and dispersive coefficients, respectively. The term $(\sqrt{\rho})_{xx}/\sqrt{\rho}$ is known as the (normalized) quantum Bohm potential \cite{Boh52a,Boh52b}, providing the model with a nonlinear third order dispersive term. It can be interpreted as a quantum correction to the classical pressure (stress tensor). The viscosity term, in contrast, is of linear type. The resulting system is used, for instance, in superfluidity \cite{Khlt89}, or in classical hydrodynamical models for semiconductor devices \cite{GarC94}.

Systems in QHD first appeared in the work by Madelung \cite{Mdlng27} as an alternative formulation of the Schr\"odinger equation, written in terms of hydrodynamical variables, and structurally similar to the Navier--Stokes equations of fluid dynamics. It constituted a precursor theory of the de Broglie--Bohm causal interpretation of quantum theory \cite{Boh52a,Boh52b,BHK87}. Since then, quantum fluid models have been applied to describe many physical phenomena, such as the mathematical description of superfluidity \cite{Khlt89,Land41}, the modeling of quantum semiconductors \cite{FrZh93,GarC94}, and the dynamics of Bose--Einstein condensates \cite{DGPS99,GrantJ73}, just to mention a few. In recent years, models in QHD have attracted the attention of mathematicians and physicists alike, thanks to their capability of describing particular physical systems, as well as their underlying mathematical challenges. Many mathematical results pertain to the existence of weak  solutions \cite{AnMa09,AnMa16,AnMa12,JMRi02,YPHQ20} and their stability \cite{AnMaZh21}, relaxation limits \cite{ACLS21}, the analysis of purely dispersive shocks \cite{GuPi74,Sgdv64,Gas01,HACCES06,HoAb07}, or the study of classical limits when the Planck constant tends to zero \cite{GasM97}. Most of these works pertain to the purely dispersive model with no viscous effects. The above list of references is by no means exhaustive and the reader is invited to see the references cited therein for more information. QHD models with viscosity have been studied under the perspective of viscous (numerical) stabilization \cite{JuMi07a}, the physical description of dense plasmas \cite{DiMu17,GMOB22,BrMe10}, the existence and stability of viscous-dispersive shocks \cite{LMZ20a,LMZ20b,FPZ22, Zhel-preprint,LaZ21b,LaZ21a,FPZ23}, the existence of global solutions \cite{GJV09}, and vanishing viscosity behaviors \cite{YaJuY15}.

In this paper, we are interested in the \emph{decay structure} of QHD systems with viscosity under the framework of Humpherys’s extension \cite{Hu05} to higher order system of the classical results by Kawashima \cite{KaTh83} and Shizuta and Kawashima \cite{ShKa85,KaSh88a} for hyperbolic-parabolic type systems. Humpherys \cite{Hu05} extended Kawashima and Shizuta's notions of strict dissipativity, genuine coupling and symmetrizability to viscous-dispersive one-dimensional systems of any order, such as the viscous QHD model under consideration. First, Humpherys introduces the concept of \emph{symbol symmetrizability}, which is a generalization of the classical notion symmetrizability in the sense of Friedrichs \cite{Frd54}, and extends the genuine coupling condition for any symmetric Fourier symbol of the linearized higher-order operator around an equilibrium state. Humpherys then shows that his notion of genuine coupling is equivalent to the strict dissipativity of the system and to the existence of a compensating matrix \emph{symbol} for the linearized system, which allows to close energy (stability) estimates.

Our analysis is divided into two parts. The first one is devoted to proving the local existence of perturbations to constant equilibrium states of the form $U_* = (\rhos, \ms)$ with $\rhos > 0$. Since the system under consideration is structurally very similar to the Navier-Stokes-Korteweg system \cite{HaLi94,Kot08} we follow the classical proof by Hattori and Li \cite{HaLi94} very closely.  Here we recast local well-posedness in terms of perturbations of an equilibrium state, a formulation which is suitable for our needs. The local existence result (see Theorem \ref{themlocale} below) guarantees the existence of classical solutions as well as the appropriate energy bounds for the perturbations which are needed for the global decay analysis. Although the arguments are classical, we present a full proof of local existence for the sake of completeness and to fulfill the requirement of obtaining appropriate energy bounds (see Corollary \ref{cor26}). Besides, there is no local existence result for the QHD system with the particular viscosity term appearing in \eqref{QHD} reported in the literature, up to our knowledge.

The second part of the paper focuses on \emph{subsonic} equilibrium states, satisfying the condition
\[
p'(\rho_*) > \frac{m_*^2}{\rho_*^2}.
\]
It can be proved (see Lemma \ref{lemsupersonic} below) that supersonic states (namely, those which satisfy $p'(\rho_*) < {m_*^2}/{\rho_*^2}$) do not satisfy the strict dissipativity condition for the linearized system, justifying in this case the choice of subsonic states for our analysis. The intermediate case of sonic states with $p'(\rho_*) = {m_*^2}/{\rho_*^2}$ is associated to degeneracies (such as in viscous-dispersive shock theory) and it is not clear whether symbol symmetrization and/or genuine coupling hold in this case. Hence, we have left the analysis of sonic states for a future work. We proceed to linearize system \eqref{QHD} around a subsonic state and to study its strict dissipativity. It is shown that the linearized QHD system is symbol symmetrizable but not Friedrichs symmetrizable, and that it satisfies the genuine coupling condition. Thanks to a new degree of freedom in the choice of the symbol symmetrizer (see also the related analyses \cite{PlV22,PlV23} for Korteweg fluids) it is possible to construct an appropriate compensating matrix symbol for the linearized system (see Lemma \ref{lemourK}), which is uniformly bounded above in the Fourier parameter and which allows to close the energy estimates at the linear level. Such estimates underlie a decay structure of \emph{regularity-gain type}, yielding optimal pointwise decay rates of the solutions to the linear system in Fourier space (see \cite{UDK12, UDK18} or Remark \ref{remtypediss} below). The linear decay rates are then used to prove the nonlinear decay of small perturbations of constant equilibrium states, culminating into the global existence and optimal time-decay of perturbation solutions (see our Main Theorem \ref{gloexth} below). 

The work of Tong and Xia \cite{ToXi22} warrants note as the first work (up to our knowledge) that analyzes the decay of perturbations of equilibrium states for a QHD system with viscosity (see also Ra and Hong \cite{RaHo21} for a similar analysis in the case of the QHD system with energy exchanges). Our work differs from the aforementioned works in the sense that their analysis consists of obtaining direct nonlinear energy estimates, relying heavily on the intrinsic structure of the QHD model. The technique presented here examines whether the linearized system around the constant state exhibits some abstract symmetrizability and dissipative properties which can be extrapolated to the nonlinear problem.

The paper is structured as follows. Section \ref{seclocale} contains the proof of local well-posedness for system \eqref{QHD}. The problem is recast in terms of perturbations or arbitrary constant equilibrium states. In Section \ref{seclinear} we study the linearized system around a subsonic equilibrium state. We examine the genuine coupling condition in the sense of Humpherys \cite{Hu05} and exhibit a family of symbol symmetrizers. The subsonicty of the constant state plays a key role. With this information, we obtain the linear decay rates for the associated semigroup. Finally, Section \ref{secglobal} contains the global decay result for small perturbations of constant equilibrium subsonic states.

\subsection*{On notation} 
Transposition of vectors or matrices is denoted by the symbol $A^\top$. Linear operators acting on infinite-dimensional spaces are indicated with calligraphic letters (e.g., $\cA$). We denote the real part of a complex number $\lambda \in \C$ by $\Re\lambda$. Standard Sobolev spaces of complex-valued functions on the real line will be denoted as $L^2(\R)$ and $H^s(\R)$, with $s \in \R$, endowed with the standard inner products and norms. The norm on $H^s(\R)$ will be denoted as $\| \cdot \|_s$ and the norm in $L^2$ will be henceforth denoted by $\| \cdot \|_0$. Any other $L^p$ -norm will be denoted as $\| \cdot \|_{L^p}$ for $\infty \geq p \geq 1$. The $L^2$-scalar product will be denoted by $\langle \cdot, \cdot \rangle_0$, whereas $\langle \cdot, \cdot \rangle$ is the standard inner product in $\C^n$. For any $T > 0$ we denote by $C_0^{\infty}([0,T] \times \R)$ the set of infinitely differentiable functions on $[0,T] \times \R$ such that $\partial_x^i f(t,\cdot) \rightarrow 0$ as $|x| \to \infty$, $i \in \mathbb{N}_0:= \N \cup \{ 0 \}$, that is, functions that for fixed $t$ have all derivatives going to zero for large $|x|$. Let $p$ be a vector in $\R^n$. Its two-norm is $|p| = \left ( \sum_{i=1}^n p_i^2 \right)^{\frac{1}{2}}$. For a matrix $A \in \R^{n \times n}$, $| A |  = \sup_{|p|=1}|A p|$ denotes its two-norm. For operators $\cA, \cB$ denote by $[\cA,\cB] = \cA \cB - \cB \cA$ their commutator. Finally, we recall Sobolev's embedding theorem: if $f \in H^1(\R)$, then $\| f\|_{L^\infty} \leq C \| f\|_{1}$, where the constant $C$ does not depend on $f$.
\label{sec:preliminaries}

\section{Local well-posedness theory}
\label{seclocale}

This section is devoted to proving the local existence of solutions for \eqref{QHD}. First we will derive energy estimates satisfied by the solution of a quasi-linearized system related to \eqref{QHD}. Then we will prove existence of solution for this system which satisfies these estimates. Finally, we will use a fixed point argument to establish local existence for the nonlinear problem. 

Let $\rhos > 0$ and $\ms \in \R$ be constant equilibrium states for system \eqref{QHD}. We are interested in proving the local existence of perturbations to this constant equilibrium state. For that purpose, we define the following space of perturbations. For positive constants $R \geq r >0$, $T>0$ and for any $s \geq 3$ we denote,
\[
\begin{aligned}
X_s((0,T);r,R) := \Big\{ (\rho,m) \, : \; \, &\rho  \in C((0,T); H^{s+1}(\R)) \cap C^1((0,T); H^{s-1}(\R)), \\
& m \in  C((0,T); H^s(\R)) \cap C^1((0,T); H^{s-2}(\R)),\\
&(\rho_x, m_x) \in L^2((0,T); H^{s+1}(\R) \times H^s(\R)), \\
&\text{and} \; r \leq \rho(x,t) \leq R \; \; \text{a.e. in } \, (x,t) \in \R \times (0,T) \Big\}.
\end{aligned}
\]
%\[
%\begin{aligned}
%X_s((0,T);r,R) := \Big\{ (\rho,m) \, : \; \, &\rho - \rhos \in C((0,T); H^{s+1}(\R)) \cap C^1((0,T); H^{s-1}(\R)), \\
%& m-\ms \in  C((0,T); H^s(\R)) \cap C^1((0,T); H^{s-2}(\R)),\\
%&(\rho_x, m_x) \in L^2((0,T); H^{s+1}(\R) \times H^s(\R)), \\
%&\text{and} \; r \leq \rho(x,t) \leq R \; \; \text{a.e. in } \, (x,t) \in \R \times (0,T) \Big\}.
%\end{aligned}
%\]
Henceforth, for any $U = (\rho, m) \in X_s\big((0,T); r,R\big)$ we denote
\begin{align}
E_s(t) &:= \sup_{\tau \in [0,t]} \Big( \| \rho(\tau)\|^2_{s+1} + \| m(\tau)\|^2_{s} \Big), 
\label{defEs}\\
F_s(t) &:= \int_0^t \Big( \| \rho_x(\tau) \|_{s+1}^2 + \| m_x(\tau) \|_s^2 \Big) \, d\tau ,
\label{defFs}
\end{align}
for all $t \in [0,T]$.  

Our main goal is to prove the following local existence result.

\begin{theorem}
\label{themlocale}
Let $\Us = (\rhos, \ms) \in \R^2$ be a constant equilibrium state with $\rhos >0$. Suppose that
\begin{equation}
\label{incond}
\rho_0  \in H^{s+1}(\R), \quad m_0  \in H^s(\R),
\end{equation}
for some $s \geq 3$ are initial perturbations of $(\rhos, \ms)$ and consider an initial condition of the form
\begin{equation}
\label{ic}
U(0)+U_* = (\rho_0 + \rhos, m_0 + \ms).
\end{equation}
Then there exists a positive constant $a_0 > 0$ such that if
\[
\| \rho_0 \|_{s+1} + \| m_0  \|_{s} < a_0,
\]
then there holds $r_0 \leq \rhos + \rho_0(x) \leq R_0$, \emph{a.e.} in $x \in \R$, for some constants $R_0 > r_0 > 0$, and there exists a positive time $T_1 = T_1(a_0) > 0$ such that a unique smooth solution of the form $(\rho(x,t) + \rhos, m(x,t) + m_*)$, with perturbation belonging to the space
\[
(\rho ,m) \in X_s\big((0,T_1); \tfrac{1}{2}r_0,2R_0\big),
\]
exists for the Cauchy problem of system \eqref{QHD} with initial data \eqref{ic}. Moreover, the solution satisfies the energy estimate
%\[
%\begin{aligned}
%\sup_{\tau \in [0,T_1]} \big( \| \rho(\tau)\|^2_{s+1} + \| m(\tau)\|^2_{s} \big) + \int_0^{T_1} \big( \| \rho_x(\tau) \|_{s+1}^2 &+ \| m_x(\tau) \|_s^2 \big) \, d\tau \\ &\leq C \| \rho_0\|^2_{s+1} + \| m_0\|^2_{s}.
%\end{aligned}
%\]
\begin{equation}
\label{localEE}
E_s(T_1) + F_s(T_1) \leq C_0 E_s(0),
\end{equation}
for some constant $C_0 > 0$ depending only on $a_0$.
\end{theorem}

\begin{corollary}[a priori estimate]
\label{cor26}
Under the assumptions of Theorem \ref{themlocale}, let $U = (\rho+\rhos,m+\ms)$ with $(\rho,m) \in X_s\big((0,T); \tfrac{1}{2}r_0,2R_0\big)$  be a local solution of the initial value problem of \eqref{QHD} with initial data $U(0)+U_* = (\rho_0+\rhos,m_0+\ms)$ satisfying \eqref{incond}. Then there exists $0 < \e_1 \leq a_0$, sufficiently small, such that, if for any $0 < t  \leq T$ we have $E_s(t)^{1/2} \leq \e_1$, then there holds
\begin{equation}
\label{localaprioriEE}
\big( E_s(t) + F_s(t) \big)^{1/2} \leq C_2 E_s(0)^{1/2},
\end{equation}
where $C_2 = C_2(\e_1) > 0$ is a positive constant independent of $t > 0$.
\end{corollary}
\begin{proof}
Follows directly from the local existence result, Theorem \ref{themlocale} (see, e.g., the proof of Theorem 2.2 in \cite{HaLi96a}).
\end{proof}

\subsection{The linear system}
Our purpose is to formulate the well-posedness for perturbations of a given constant state. Hence, consider deviations from the given state $(\rc, \mc)$, which will be denoted as
\begin{equation*}
\rho = \br + \rc, \qquad m = \brm + \mc.
\end{equation*}
The system \eqref{QHD} in the new perturbation variables $(\br, \brm)$ reads
\begin{equation}
\label{QHD-L-deviation}
	\begin{aligned}
		\br_t+\brm_x &= 0,\\
		\brm_t+\left(\displaystyle\frac{(\brm + \mc)^2}{\br + \rc}+p(\br + \rc)\right)_x &=\mu \brm_{xx}+k^2(\br + \rc)\left(\displaystyle\frac{(\sqrt{\br + \rc})_{xx}}{\sqrt{\br + \rc}}\right)_x.
	\end{aligned}
\end{equation}

Now we deduce a linearized system that will be useful for proving local existence of solutions. For $\br, \rho: \R \rightarrow \R^+$ and $\brm,m :\R \rightarrow \R$ denote $w = (\rho, m)^\top, \bw = (\br, \brm)^\top\in \R^2$. Let $T > 0$ and assume that
\begin{align}
\label{cond_bound}
\sup_{x \in \R, \, t \in [0,T]} \left( \frac{1}{\br + \rc} + \sum_{i=0}^2 |\partial_x^i \bw| \right) \leq \beta_0,
\end{align}
for some constant $\beta_0 > 0$. For such $\bw$ we define
\begin{align*}
\balp(x,t) &= \left (\frac{\brm(x,t) + \mc}{\br(x,t) + \rc} \right )^2 - \gamma (\br(x,t) + \rc)^{\gamma - 1},\\
\bbet(x,t) &= -\frac{2 (\brm(x,t) + \mc)}{\br(x,t) + \rc},
\end{align*}
and the matrix
\begin{align*}
\bA(x,t) = \begin{pmatrix}
0 & -1\\
\balp(x,t) & \bbet(x,t)
\end{pmatrix}.
\end{align*}
Now, let us denote
\begin{equation*}
\zeta = k^2 \dfrac{\br_x}{\br + \rc},\qquad \eta = \dfrac{k^2}{2} \left ( \dfrac{\br_x}{\br + \rc} \right )^2.
\end{equation*}
Moreover, define the operators,
\begin{align*}
&\cT_1 w = \begin{pmatrix}
0 \\ \frac{k^2}{2} \rho_{xxx}
\end{pmatrix},
&&\cT_2 w = \begin{pmatrix}
0 \\ \mu m_{xx}
\end{pmatrix},\\
&\cT_3 w = \begin{pmatrix}
0 \\ - \zeta \rho_{xx}
\end{pmatrix},
&&\cT_4 w = \begin{pmatrix}
0\\ \eta \rho_x
\end{pmatrix},
\end{align*}
and let
\begin{equation}
\label{operator_L}
\cL w := \bA w_x + \sum_{i=0}^4 \cT_i w.
\end{equation}
For any vector valued function $f = (f_1,f_2)^\top$, we consider the following linear system
\begin{equation}
\label{linear_system}
\left\{
\begin{aligned}
\partial_t w &= \cL w + f,\\
w(0) &= w_0.
\end{aligned}
\right.
\end{equation}
This system has the property that if $w = \bw$ and $f = 0$, then $\bw$ solves \eqref{QHD-L-deviation}. Moreover, we denote
\begin{equation}
\label{eq:components_L}
\cL w = \begin{pmatrix} \cL_1 w\\ \cL_2 w \end{pmatrix}.
\end{equation}

\subsection{The zeroth order estimate}

Here we derive an estimate satisfied by smooth solutions to the linear system \eqref{linear_system}. For that purpose, let us define the norm 
\begin{equation*}
\vertiii{w}_{[0,T]}^2 := \sup_{t \in [0,T]} (\| w(t) \|_0^2 + \| \partial_x \rho(t) \|_0^2) + \int_{0}^T  \| \partial_x m(t) \|_0^2 \, dt,
\end{equation*}
for any arbitrary $T > 0$.
\begin{lemma}[zeroth order estimate]
\label{theorem_zero_order}
Suppose that $\bw$ satisfies the bound \eqref{cond_bound} and $w, f \in C_0^{\infty}([0,T] \times \R)$ for some $T > 0$. Then the following estimate holds for $t \in (0,T)$,
\begin{equation}
\label{eq:estimate_0_order}
\begin{aligned}
\partial_t\left ( \frac{1}{2} \| w \|_0^2 + \frac{k^2}{4} \|\rho_x\|_0^2 \right ) &+ \left (\mu - \frac{C_1 \e_1}{2} - \frac{C_3 \e_2}{2} \right) \|m_x\|_0^2 \\
&\leq \frac{1}{2} \left (\frac{C_1}{\e_1} + k^2 C_2 + C_4 + 1 \right) \| w \|_0^2 + \\
&\quad + \frac{1}{2} \left( \frac{k^2}{2} + C_1 \e_1 + k^2 C_2 + \frac{C_3}{\e_2} + C_4 \right) \| \rho_x\|_0^2 +  \\ 
&\quad + \frac{1}{2}\| f \|_0^2  + \frac{k^2}{4}\| f_1 \|_1^2, 
\end{aligned}
\end{equation}
for any $\e_1, \e_2 >0$ such that
\begin{equation*}
\mu - \frac{C_1 \e_1}{2} - \frac{C_3 \e_2}{2} > 0,
\end{equation*}
and explicit constants $C_i> 0,\mbox{ }i=1,...,4$ depending only on $\beta_0$, $\rc$, $\mc$ and the physical parameters of \eqref{QHD}. Moreover,
\begin{equation}
\vertiii{w}_{[0,T]}^2 \leq C(T) \Big( \| w_0\|_0^2 + \| \rho_0 \|_1^2 + \int_0^T \big(\|f(s)\|_0^2 + \|f_1(s)\|_1^2\big) \, ds \Big).
\label{eq:integ_estimate_0_order}
\end{equation}
The constant $C(T)$ in \eqref{eq:integ_estimate_0_order} depends only on $T$, $\beta_0$ and the parameters of \eqref{QHD}.
\end{lemma}
\begin{proof}
Taking the $L^2$-scalar product of \eqref{linear_system} with $w$, we get
\begin{equation}
\label{eq:scalar_product_0_order}
\langle \partial_t w, w \rangle_0  = \langle \bA w_x, w \rangle_0 + \left \langle \sum_{i=0}^4 \cT_i w, w \right\rangle_0 + \langle f, w \rangle_0.
\end{equation}
We also have
\begin{align*}
\langle \partial_t w, w \rangle_0 = \int_\R \left ( \rho \partial_t \rho  +  m \partial_t m\right )\, dx 
= \tfrac{1}{2} \int_\R \left ( (\rho^2)_t + (m^2)_t \right ) \, dx = \tfrac{1}{2}\partial_t(\| w\|_0^2).
\end{align*}
Moreover,
\begin{equation*}
|(\bA w_x)(x,t)| \leq |\bA(x,t)| |w_x(x,t)|,\qquad \text{for all } (x,t) \in \R \times [0,T].
\end{equation*}
Since $\sup_{x \in \R, t \in [0,T]} (\rho + \rc)^{-1}(x,t) \leq \beta_0$ we can find a positive constant $C_1$, depending only on $\beta_0$, $\rc$ and $\mc$, such that
\begin{equation}
\sup_{x \in \R, \, t \in [0,T]} |\bA(x,t)| \leq C_1.
\end{equation}
Applying Young's inequality we deduce
\begin{equation}
\label{eq:est_A}
\begin{aligned}
\langle \bA w_x, w \rangle_0 = \int_{\R} (\bA w_x)^\top w \, dx &\leq \int_\R | \bA| |w_x| |w| \, dx \\
&\leq C_1 \int_\R |w_x| |w| \, dx \\
&\leq C_1 \| w_x\|_0 \| w \|_0 \\
&\leq \frac{C_1}{2 \e_1}\| w\|_0^2 + \frac{C_1 \e_1}{2}\|w_x\|_0^2  \\
%&= \frac{C_1}{2 \e_1}\| w\|_0^2 + \frac{C_1 \e_1}{2}\left ( \int_\R (\rho_x)^2 \, dx + \int_\R (m_x)^2 \, dx \right)  \\
&= \frac{C_1}{2 \e_1}\| w\|_0^2 + \frac{C_1 \e_1}{2} \| \rho_x \|_0^2 + \frac{C_1 \e_1}{2} \|m_x\|_0^2. 
\end{aligned}
\end{equation}
Denote $I = \langle \cT_1 w, w \rangle_0$. Then, upon integration by parts, one gets
\begin{align}\label{eq_dispersive}
I = \frac{k^2}{2} \int_\R \rho_{xxx} m \, dx = -\frac{k^2}{2} \int_\R \rho_{xx} m_x \, dx.
\end{align}
The first component of \eqref{linear_system} implies that $-m_x = \rho_t - f_1$. Substituting into \eqref{eq_dispersive} we get
\begin{align}
\label{eq:util1}
I = \frac{k^2}{2} \int_\R (\rho_t - f_1) \rho_{xx} \, dx = -\frac{k^2}{2} \int_\R \rho_{tx} \rho_x \, dx +
\frac{k^2}{2}\int_\R  (f_1)_x \rho_x \, dx
\end{align}
We clearly have
\[
-\frac{k^2}{2} \int_\R \rho_{tx} \rho_x \, dx = -\frac{k^2}{4} \partial_t \int_\R (\rho_x)^2 \, dx = -\frac{k^2}{4} \partial_t(\|\rho_x\|_0^2),
\]
and,
\[
\frac{k^2}{2}\int_\R  (f_1)_x \rho_x \, dx \leq \frac{k^2}{2} \| f_1 \|_1 \| \rho_x \|_0 \leq \frac{k^2}{4} \|f_1\|_1^2 + \frac{k^2}{4} \| \rho_x \|_0^2.
\]
Substituting back into \eqref{eq:util1} we obtain
\begin{equation}
\label{eq:est_T1}
\langle \cT_1 w, w \rangle_0 \leq -\frac{k^2}{4} \partial_t(\|\rho_x\|_0^2) + \frac{k^2}{4} \|f_1\|_1^2 + \frac{k^2}{4} \| \rho_x \|_0^2.
\end{equation}
Now, integrate by parts to get
\[
\langle \cT_2 w, w \rangle_0 = \mu \int_\R m m_{xx} \, dx = -\mu \| m_x \|_0^2.
\]
We also have, after integration by parts, that
\begin{align*}
\langle \cT_3 w, w \rangle_0 &= - \int_\R \zeta \rho_{xx} m \, dx =  \int_\R \zeta_x m \rho_x \, dx + \int_\R \zeta m_x \rho_x \, dx.
\end{align*}

Observe that $\zeta_x$ and $\zeta$ involve only derivatives up to second and first order, respectively. Hence, due to the bound \eqref{cond_bound}, we can find positive constants $C_2$ and $C_3$ depending only on $\beta_0$, such that
\begin{align*}
\sup_{x \in \R, \, t \in [0,T]} | \zeta_x | \leq C_2, \qquad
\sup_{x \in \R, \, t \in [0,T]} | \zeta | \leq C_3.
\end{align*}
Henceforth, we obtain
\begin{align}
\langle \cT_3 w, w \rangle_0 &\leq  C_2 \int_\R  |m| |\rho_x| \, dx + C_3 \int_\R  |m_x| |\rho_x| \, dx \nonumber\\
&\leq C_2 \| m \| \|_0 \rho_x \|_0 + C_3 \| m_x \|_0 \| \rho_x \|_0 \nonumber\\
&\leq \frac{C_2}{2} \| m \|_0^2 + \frac{C_2}{2} \| \rho_x \|_0^2 + \frac{C_3}{2 \e_2} \| \rho_x \|_0^2 + \frac{C_3 \e_2}{2} \| m_x \|_0^2 \nonumber\\
&\leq \frac{C_2}{2} \| w \|_0^2 + \frac{C_2}{2} \| \rho_x \|_0^2 + \frac{C_3}{2 \e_2} \| \rho_x \|_0^2 + \frac{C_3 \e_2}{2} \| m_x \|_0^2. \label{eq:est_T3}
\end{align}
Moreover, from the definition of $\cT_4$ we have
\begin{align*}
\langle \cT_4 w, w \rangle_0 = \int_\R \eta \rho_x m \, dx.
\end{align*}
Denote
\begin{equation*}
C_4(T) := \sup_{x \in \R, \, t \in [0,T]}|\eta|.
\end{equation*}
Then we have
\begin{align}
\langle \cT_4 w, w \rangle_0 &\leq C_4 \int_\R |\rho_x| |m| \, dx \nonumber\\
%&\leq C_4 \| \rho_x \|_0 \| m \|_0 \nonumber\\
&\leq \frac{C_4}{2} \| \rho_x \|_0^2 + \frac{C_4}{2} \| m \|_0^2 \nonumber\\
&\leq \frac{C_4}{2} \| \rho_x \|_0^2 + \frac{C_4}{2} \| w \|_0^2. \label{eq:est_T4}
\end{align}
Finally, from $\langle f, w \rangle_0 \leq \| f \|_0 \| w\|_0 \leq \frac{1}{2} \| f \|_0^2 + \frac{1}{2} \|w\|_0^2$ and substituting  \eqref{eq:est_A}, \eqref{eq:est_T1} - \eqref{eq:est_T4} into \eqref{eq:scalar_product_0_order}, we arrive at estimate \eqref{eq:estimate_0_order}.

Now, for any $c \in (0,1)$, let us choose $\e_1$ and $\e_2$ such that
\begin{equation*}
\frac{C_1 \e_1}{2} +  \frac{C_3 \e_2}{2}  \leq (1 - c) \mu.
\end{equation*}
We obtain
\begin{equation}
\partial_t \Big( \tfrac{1}{2} \| w \|_0 ^2 + \tfrac{1}{4} k^2 \|\rho_x\|_0^2 \Big) + c \mu \| m_x \|_0^2 \leq C \Big( \tfrac{1}{2}\| w \|_0 ^2 + \tfrac{1}{4} k^2 \|\rho_x\|_0^2 +\| f\|_0^2 + \| f_1 \|_1^2 \Big), \label{eq:estimate_0_order_2}
\end{equation}
where $C$ depends only on $\beta_0$. Then, inequality \eqref{eq:estimate_0_order_2} clearly implies
\begin{align}\label{eq:util4}
\partial_t \Big( \tfrac{1}{2} \| w \| ^2 + \tfrac{1}{4} k^2 \|\rho_x\|^2 \Big) \leq C \Big( \tfrac{1}{2}\| w \| ^2 + \tfrac{1}{4}k^2 \|\rho_x\|^2 +\| f\|^2 + \| f_1 \|_1^2 \Big).
\end{align}
Apply Gronwall's inequality to \eqref{eq:util4} in order to obtain
\begin{align}
\tfrac{1}{2}\| w(t)\|_0^2  + \tfrac{1}{4} k^2\| \partial_x \rho(t)\|_0^2 &\leq e^{C t} \left (  \tfrac{1}{2} \| w_0 \|_0^2 + \frac{1}{4} k^2 \| \partial_x \rho_0 \|_0^2 \right ) + \nonumber\\
&\quad + C \int_{0}^T e^{C(t-s)} \Big(\| f(s) \|_0^2 + \| f_1(s) \|_1^2 \Big) \, ds, \label{gronwall1}
\end{align}
for all $t \in [0,T]$ and where $C = C(T)>0$ depends only on $T$ and $\beta_0$. Substituting \eqref{gronwall1} into \eqref{eq:estimate_0_order_2} and integrating from $0$ to $\tilde{t}$ yields \eqref{eq:integ_estimate_0_order}.
\end{proof}

\subsection{Higher order estimates}

For $n \in \N$, consider the norm
\begin{equation*}
\vertiii{w}_{n,[0,T]}^2 := \sup_{t \in [0,T]} \Big(\| \rho(t) \|_{n+1}^2 + \| m(t) \|_n^2\Big) + \int_{0}^T  \|m_x(t)\|_n^2 \, dt.
\end{equation*}
The following result establishes estimates of higher order on the solutions.
\begin{lemma}[higher order estimate]
Let $n \in \N$, $T > 0$ and $c \in (0,1)$. Suppose 
\begin{align}
\label{cond_bound_higher_order}
\sup_{x \in \R, \, t \in [0,T]} \left( \frac{1}{\br + \rc} + \sum_{i=0}^2 |\partial_x^i \bw| \right) + \vertiii{\bw}_{n,[0,T]}^2 \leq \beta_n,
\end{align}
where $\beta_n > 0$ is a constant.
Then, we have
\begin{equation}
\begin{aligned}
\pd_t \Big( \tfrac{1}{2} \| w \|_n^2 + \tfrac{1}{4} k^2 \| \rho \|_{n+1}^2 \Big) + c \mu \| m \|_{n+1}^2 \leq C_n \Big( \tfrac{1}{2} \| w \|_n^2 + \tfrac{1}{4} k^2 \| \rho \|_{n+1}^2 + \| f \|_n^2 + \| f_1 \|_{n+1}^2 \Big),
\label{estimate_higher_order}
\end{aligned}
\end{equation}
for all $t \in (0,T)$, and
\begin{equation}
\vertiii{w}_{n,[0,T]}^2 \leq C_n(T) \big( \| w_0 \|_n^2 + \| \rho_0 \|_{n+1}^2 \big) + C_n(T) \int_{0}^{T} \!\big( \| f(s) \|_n^2 + \| f_1(s) \|_{n+1}^2 \big) \, ds, \label{integral_estimate_higher_order}
\end{equation}
where $C_n$ depends only on $T$, $\beta_n$, $\rc$, $\mc$ and the parameters of \eqref{QHD}.
\end{lemma}
\begin{proof}
Differentiating \eqref{linear_system} we obtain that the equation satisfied by $\pd^j_x w$ is
\begin{equation}
\label{eq:higher_order_sys}
\left\{
\begin{aligned}
(\pd_x^j w)_t &= \cL \pd^j_x w +[\pd^j_x,\cL] w + \pd^j_x f,\\
\pd_x^j w(0) &= \pd_x^j w_0.
\end{aligned}
\right.
\end{equation}
for $j \in \N_0$. Suppose $j \in \{0,...,n\}$ and take the scalar product of \eqref{eq:higher_order_sys} with $\pd_x^j w$. The result is
\begin{equation*}
\langle (\pd_x^j w)_t, \pd_x^j w \rangle_0 = \langle \cL \pd^j_x w, \pd_x^j w \rangle_0 +\langle  [\pd^j_x,\cL] w, \pd_x^j w \rangle_0 + \langle  \pd^j_x f, \pd_x^j w \rangle_0.
\end{equation*}
Clearly we have
\[
\langle (\pd_x^j w)_t, \pd_x^j w \rangle_0 = \tfrac{1}{2} \pd_t (\| \pd_x^j w \|_0^2).
\]
Moreover,
\begin{equation*}
\cL \partial_x^j w = \bA \pd_x^{j+1} w + \sum_{i=1}^{4} \cT_i(\pd_x^j w). 
\end{equation*}
In the sequel, $C$ denotes a positive constant that may depend only on $\beta_n$, $\rc$, $\mc$ and on the parameters of \eqref{QHD}, and which may change from line to line. We have
\begin{align}
\langle  \bA \pd_x^{j+1} w, \pd_x^j w \rangle_0 &\leq C \| \pd_x^{j+1} w \|_0 \| \pd_x^j w \|_0 \nonumber\\
&\leq C \e_1 \| \pd_x^{j+1} w \|_0^2 + C(\e_1) | \pd_x^j w \|_0^2 \nonumber \\
&\leq C \e_1 \| \pd_x^{j+1} \rho \|_0^2 + C \e_1 \| \pd_x^{j+1} m \|_0^2 + C(\e_1) | \pd_x^j w \|_0^2. \label{estimate_A}
\end{align}
Now we shall estimate the contribution from the dispersive term. Integrating by parts we get
\begin{equation}
\label{eq:dispersive_higher_order}
\langle \cT_1 \pd_x^j w, \pd_x^j w \rangle_0 = \frac{k^2}{2} \int_\R (\pd_x^j m) (\pd_x^{j+3} \rho) \, dx
= - \frac{k^2}{2} \int_\R (\pd_x^{j+1} m) (\pd_x^{j+2} \rho) \, dx.
\end{equation}
The first component of \eqref{eq:higher_order_sys} reads (see equation \eqref{eq:components_L}),
\begin{equation*}
(\pd_x^j \rho)_t = - \pd_x^{j+1} m + [\pd_x^j, \cL_1] w + \pd_x^j f_1.
\end{equation*}
Since $[\pd_x^j, \cL_1] w = 0$, we obtain $\pd_x^{j+1} m = -(\pd_x^j \rho)_t + \pd_x^j f_1$. Substituting into \eqref{eq:dispersive_higher_order} we get
\begin{align}
\langle \cT_1 \pd_x^j w, \pd_x^j w \rangle_0 = I_1 + I_2, \label{eq:T1}
\end{align}
where
\[
I_1 := \frac{k^2}{2}\int_\R (\pd_x^j \rho)_t (\pd_x^{j+2} \rho) \, dx, \quad I_2 := - \frac{k^2}{2}\int_\R (\pd_x^j f) (\pd_x^{j+2} \rho) \, dx.
\]
First, notice that
\[
I_1 = -\frac{k^2}{2}\int_\R \pd_x ( (\pd_x^j \rho)_t ) (\pd_x^{j+1} \rho) \, dx = -\frac{k^2}{2}\int_{\R} \pd_t (\pd_x^{j+1} \rho) (\pd_x^{j+1} \rho) \, dx = - \frac{k^2}{4} \pd_t \| \pd_x^{j+1} \rho \|_0^2. 
\]
On the other hand, we can estimate $I_2$ by
\[
I_2 = -\frac{k^2}{2}\int_\R (\pd_x^{j+1} f) (\pd_x^{j+1}\rho) \, dx \leq C \|f_1\|_{j+1}^2 + C \|\pd_x^{j+1}\rho\|_0^2. 
%\label{eq:I2}
\]
Substitution into \eqref{eq:T1} yields
\begin{equation}
\label{estimate_T1}
\langle \cT_1 \pd_x^j w, \pd_x^j w \rangle_0 \leq - \tfrac{1}{4} k^2 \pd_t \| \pd_x^{j+1} \rho \|_0^2
+ C \|f_1\|_{j+1}^2 + C \|\pd_x^{j+1}\rho\|_0^2.
\end{equation}
Now, let us consider $\langle \cT_2 \pd_x^j w, \pd_x^j w\rangle_0$. Integrate by parts in order to obtain,
\[
\langle \cT_2 \pd_x^j w, \pd_x^j w \rangle_0 = \mu \int_{\R} (\pd_x^j m) (\pd_x^{j+2} m) dx = -\mu \int_\R (\pd_x^{j+1} m)^2 \, dx = -\mu \| \pd_x^{j+1} m \|_0^2. 
%\label{eq:T2}
\]
Then, we deduce
\[
\langle \cT_3 \pd_x^j w, \pd_x^j w \rangle_0= - \int_{\R} \zeta (\pd_x^j m) (\pd_x^{j+2} \rho) \, dx = \int_\R  \zeta_x (\pd_x^j m) (\pd_x^{j+1} \rho) \, dx
+ \int_\R \zeta (\pd_x^{j+1} m) (\pd_x^{j+1} \rho) \, dx.
\]
Notice that $\left \| \zeta \right \|_{L^\infty}$, $\left \| \zeta_x \right \|_{L^\infty}$ and $\left \| \eta \right \|_{L^\infty}$ are bounded by a constant depending only on $\beta_n$ and $\rc$. Therefore,
\begin{equation}
\label{estimate_T3}
\begin{aligned}
\langle \cT_3 \pd_x^j w, \pd_x^j w \rangle_0 &\leq \| \zeta_x \|_{L^\infty}  
\int_\R | \pd_x^j m | | \pd_x^{j+1} \rho | \, dx + \| \zeta \|_{L^\infty} \int_\R | \pd_x^{j+1} m | | \pd_x^{j+1} \rho | \, dx \\
&\leq C \| \pd_x^j m \|_0 \| \pd_x^{j+1} \rho \|_0 + C \| \pd_x^{j+1} m \|_0 \| \pd_x^{j+1} \rho \|_0 \\
&\leq C \| \pd_x^j m \|_0^2 + \frac{C}{\e_2} \|\pd_x^{j+1} \rho \|_0^2 + C \e_2 \| \pd_x^{j+1} m\|_0^2. 
\end{aligned}
\end{equation}
Now,
\begin{align}
\langle \cT_4 \pd_x^j w, \pd_x^j w \rangle_0
&= \int_{\R} \eta (\pd_x^j m) (\pd_x^{j+1}\rho) \, dx \nonumber\\
&\leq \| \eta \|_{L^\infty} \!\int_\R |\pd_x^j m| |\pd_x^{j+1}\rho| \, dx \nonumber\\
%&\leq C \| \pd_x^j m \|_0 \| \pd_x^{j+1}\rho \|_0 \nonumber\\
&\leq C \big( \| \pd_x^j m \|_0^2 +  \| \pd_x^{j+1}\rho \|_0^2 \big). \label{estimate_T4}
\end{align}
Furthermore,
\[
\langle \pd_x^j f, \pd_x^j w \rangle_0 \leq \| \pd_x^j f \|_0 \| \pd_x^j w\|_0 \leq C \big( \| \pd_x^j f \|_0^2 +  \| \pd_x^j w \|_0^2 \big).
\]

Now, let us consider the contribution from the term involving the commutator, which reads $\langle [\pd_x^j, \cL] w, \pd_x^j w \rangle_0$. If $j = 0$, we have that $[\pd_x^j, \cL]w = 0$. Therefore we will assume that $j \geq 1$. Notice that we also have $[\pd_x^j, \cL_1] w = 0$. Let us first prove a statement that will be used later.
If $i \in \{0,...,j-1\}$, then $\pd_x^{j-i}\balp$ and $\pd_x^{j-i}\bbet$ contain derivatives up to order $j$ of $\br$ and $\brm$. Moreover, $\pd_x^{j-i}\zeta$ and $\pd_x^{j-i}\eta$ contain derivatives up to order $j+1$ or $\br$. Therefore,
\begin{equation*}
\| \pd_x^{j-i}\balp \|_0\mbox{, }\| \pd_x^{j-i}\bbet \|_0 \mbox{, }\| \pd_x^{j-i}\zeta \|_0 \mbox{, }\| \pd_x^{j-i}\eta \|_0
\leq C(\beta_n, \rc, \mc),
\end{equation*}
for $i \in \{0,...,j-1\}$. By the Leibnitz formula we have
\begin{equation*}
\pd_x^j(\balp \pd_x \rho) = \sum_{i=0}^j {j \choose i} (\pd_x^{j-i} \balp) (\pd_x^{i+1} \rho).
\end{equation*}
Let us denote $a_1 =  \pd_x^j(\balp \pd_x \rho) - \balp \pd_x^{j+1}\rho $. Then,
\[
\begin{aligned}
\left \| a_1 \right \|_0
= \left \| \sum_{i=0}^{j-1} {j \choose i} (\pd_x^{j-i} \balp) (\pd_x^{i+1} \rho) \right \|_0 &\leq C \sum_{i=0}^{j-1}  \| (\pd_x^{j-i} \balp) (\pd_x^{i+1} \rho)  \|_0 \\&\leq C \sum_{i=0}^{j-1}  \| \pd_x^{j-i} \balp \|_0 \| \pd_x^{i+1} \rho \|_{L^\infty} .
\end{aligned}
\]
Also, by the Sobolev embedding theorem, we have the estimate 
\begin{equation*}
\| \partial_x^{i+1} \rho \|_{L^\infty} \leq C \| \rho \|_{j+1},\qquad i \in \{0,...,j-1\}.
\end{equation*}
Therefore, $\left \| a_1 \right \|_0 \leq C \| \rho \|_{j+1}$ and
%\begin{equation*}
%\left \| a_1 \right \|_0 \leq C \| \rho \|_{j+1}.
%\end{equation*}
%Hence,
\begin{equation}
\langle \pd_x^j(\balp \pd_x \rho) - \balp \pd_x^{j+1}\rho, \pd_x^j m \rangle_0 \leq C \| \rho \|_{j+1} \| m \|_j
\leq C \| \rho \|_{j+1}^2 + C \| m \|_{j}^2. \label{estimate_commutator1}
\end{equation}
Now, denote $a_2 = \pd_x^j(\bbet \pd_x m) - \bbet \pd_x^{j+1}m$. We clearly have the estimate
\[
\begin{aligned}
\left \| a_2 \right \|_0
= \left \| \sum_{i=0}^{j-1} {j \choose i} (\pd_x^{j-i} \bbet) (\pd_x^{i+1} m) \right \|_0 &\leq C \sum_{i=0}^{j-1}  \| (\pd_x^{j-i} \bbet) (\pd_x^{i+1} m)  \|_0
\\
&\leq C \sum_{i=0}^{j-1}  \| \pd_x^{j-i} \bbet \|_0  \| \pd_x^{i+1} m \|_{L^\infty}.
\end{aligned}
\]
Since we have $\| \partial_x^{i+1} m \|_{L^\infty} \leq C \| m \|_{j+1}$ for all $i \in \{0,...,j-1\}$ then $\left \| a_2 \right \|_0 \leq C \| m \|_{j+1}$ and
\begin{equation}
\label{estimate_commutator2}
\langle \pd_x^j(\bbet \pd_x m) - \bbet \pd_x^{j+1}m, \pd_x^j m \rangle_0 \leq C \| m \|_{j+1} \| m \|_j \leq C \e_3 \| m \|_{j+1}^2 + \frac{C}{\e_3} \| m \|_j^2.
\end{equation}
Observe that $[\pd_x^j, \cT_1] = [\pd_x^j, \cT_2] = 0$. Let us now estimate the contribution from $\cT_3$:
\begin{align*}
\langle [\pd_x^j, \cT_3]\rho, \pd_x^j m \rangle_0 &=
- \int_\R \left ( \pd_x^j ( \zeta \pd_x^2 \rho) - \zeta \pd_x^{j+2}\rho \right ) (\pd_x^j m) \, dx \\
&= - \sum_{i=0}^{j-1} {j \choose i} \int_\R (\pd_x^{j-i}\zeta) (\pd_x^{i+2} \rho) (\pd_x^j m) \, dx \\
&\leq C \sum_{i=0}^{j-1} \| \pd_x^{j-i}\zeta \|_0 \| (\pd_x^{i+2} \rho) (\pd_x^j m )\|_0 \\
&\leq C \sum_{i=0}^{j-1} \| (\pd_x^{i+2} \rho)( \pd_x^j m) \|_0 \\
&\leq C \sum_{i=0}^{j-1} \| \pd_x^{i+2} \rho \|_0 \| \pd_x^j m \|_{L^\infty}.
\end{align*}
\begin{equation*}
\end{equation*}
We have $i \in \{ 0,...,j-1 \}$ or, equivalently, $i+2 \in \{2,...,j+1\}$. Therefore $\| \pd_x^{i+2} \rho \| \leq \| \rho\|_{j+1}$. Moreover, $\| \pd_x^j m \|_{L^\infty} \leq C \| m \|_{j+1}$. Whence we get
\begin{equation}
\langle [\pd_x^j, \cT_3]\rho, \pd_x^j m \rangle_0 \leq C \| \rho\|_{j+1} \| m \|_{j+1}
\leq \frac{C}{\e_4} \| \rho \|_{j+1}^2 + C \e_4 \| m \|_{j+1}^2. \label{estimate_commutator_T3}
\end{equation}
Finally, let us consider the contribution from $\cT_4$. We have 
\begin{align*}
\langle [\pd_x^j, \cT_4]\rho, \pd_x^j m \rangle_0 &=
\int_\R \left ( \pd_x^j ( \eta \pd_x \rho) - \eta \pd_x^{j+1}\rho \right ) (\pd_x^j m) \, dx \\
&= \sum_{i=0}^{j-1} {j \choose i} \int_\R (\pd_x^{j-i}\eta) (\pd_x^{i+1} \rho) (\pd_x^j m) \, dx\\
&\leq C \sum_{i=0}^{j-1} \| \pd_x^{j-i}\eta \|_0 \| (\pd_x^{i+1} \rho)( \pd_x^j m) \|_0\\
&\leq C \sum_{i=0}^{j-1} \| (\pd_x^{i+1} \rho)( \pd_x^j m) \|_0 \\
&\leq C \sum_{i=0}^{j-1} \| \pd_x^{i+1} \rho \|_{L^\infty} \| \pd_x^j m \|_0.
\end{align*}
Since $i \in \{ 0, ..., j-1 \}$, we have $i+1 \in \{1, ..., j \}$ and $\| \pd_x^{i+1} \rho \|_{L^\infty} \leq \| \pd_x^{j} \rho \|_{L^\infty} \leq C \| \rho \|_{j+1}$.
Moreover, there holds $\| \pd_x^j m \|_0 \leq \| m \|_{j}$. This implies that,
\begin{equation}
\left | \langle [\pd_x^j, \cT_4]\rho, \pd_x^j m \rangle_0 \right | \leq C \| \rho \|_{j+1} \| m \|_{j}
\leq C \big( \| \rho\|_{j+1}^2 +  \| m \|_{j}^2 \big). 
\label{estimate_commutator_T4}
\end{equation}
The first component of \eqref{linear_system} is $\rho_t = - m_x + f_1$. Taking the scalar product of this equation with $k^2 \rho/2$ we infer
\begin{equation}
\label{energy_estimate_continuity_equation}
\frac{k^2}{4} \pd_t ( \| \rho \|_0^2 ) \leq C \left ( \| \rho \|_0^2 + \| \rho_x \|_0^2 + \| m \|_0^2 + \| f_1 \|_0^2 \right ).
\end{equation}

Let $c \in (0,1)$ be an arbitrary constant. Using estimates \eqref{estimate_A}, and \eqref{estimate_T1}  thru \eqref{energy_estimate_continuity_equation}, summing for 
$j \in \{0,...,n\}$, and choosing $\e_i > 0$, $i \in \{1,...,4\}$ sufficiently small, we obtain
\[
\pd_t \Big( \tfrac{1}{2} \| w \|_n^2 + \tfrac{1}{4} k^2 \| \rho \|_{n+1}^2 \Big) + c \mu \| m \|_{n+1}^2 \leq C \Big( \| w \|_n^2 + \| \rho\|_{n+1}^2 + \| f \|_n^2 + \| f_1 \|_{n+1}^2 \Big). 
\]
This last inequality implies \eqref{estimate_higher_order}. Similarly to the proof of Lemma \ref{theorem_zero_order}, applying Gronwall's inequality to \eqref{estimate_higher_order} we obtain \eqref{integral_estimate_higher_order}.
\end{proof}

\subsection{Existence of solutions to the linear system}

Let $\phi,\mbox{ }g: \R \rightarrow \R^2$. The formal adjoint of operator \eqref{operator_L} reads
\label{eq:components_adj_L}
\begin{equation}
\cL^* \phi = \begin{pmatrix} \cL_1^* \phi\\ \cL_2^*\phi \end{pmatrix},
\end{equation}
where
\begin{align*}
&\cL_1^*\phi = - (\balp \phi_2)_x - \frac{k^2}{2} (\phi_2)_{xxx} - ( \zeta \phi_2 )_{xx} -  
( \eta \phi_2 )_x,\\
&\cL_2^*\phi = (\phi_1)_x - (\bbet \phi_2)_x + \mu (\phi_2)_{xx},
\end{align*}
and the associated adjoint system is
\begin{equation}
\label{eq:adjoint_system}
\left\{
\begin{aligned}
-\pd_t \phi &= \cL^* \phi + g,\\
\phi(T) &= 0. 
\end{aligned}
\right.
\end{equation}
\subsubsection{Energy estimate for the adjoint system}
Now, we derive an estimate for the solutions to \eqref{eq:adjoint_system}.
\begin{lemma}[energy estimate]
Suppose $\bw, g \in C_0^\infty([0,T] \times \R)$ and
\begin{equation*}
\sup_{x \in \R, \, t \in [0,T]} \frac{1}{\br(x,t) + \rc} \leq C_0.
\end{equation*}
Then, we have
\begin{equation}
\label{estimate_adjoint_system}
\| \phi(t)\|_0^2 + \| \pd_x \phi_2(t) \|_0^2 \leq C(T) \int_{0}^{T} \| g(s) \|_0^2 \, ds,
\end{equation}
for all $t \in [0,T]$, where $C(T)$ depends on $C_0$, $\| \br \|_4$, $\rc$, $\| \brm \|_3$, $\mc$, $T$ and on the physical parameters of \eqref{QHD}.
\end{lemma}
\begin{proof}
The second component of \eqref{eq:adjoint_system} reads
\begin{equation}
\label{eq:second_component_adjoint_sys}
-\pd_t \phi_2 = (\phi_1)_x - (\bbet \phi_2)_x + \mu (\phi_2)_{xx} + g_2.
\end{equation}
Take the $L^2$-scalar product of \eqref{eq:second_component_adjoint_sys} with $\phi_2$. The result is
\begin{equation*}
-\langle \pd_t \phi_2, \phi_2 \rangle_0 = \langle (\phi_1)_x, \phi_2 \rangle_0 - \langle (\bbet \phi_2)_x, \phi_2 \rangle_0 + \mu \langle (\phi_2)_{xx}, \phi_2 \rangle_0 + \langle g, \phi_2 \rangle_0,
\end{equation*}
Then,
\begin{equation}
\label{eq:adjoint_1}
- \langle \pd_t \phi_2, \phi_2 \rangle_0 = -\tfrac{1}{2} \pd_t (\| \phi_2 \|_0^2).
\end{equation}
In what follows $C$ denotes a constant that may depend only on $\bw$ and on the parameters of \eqref{QHD} and that may change from line to line. We have,
\[
\begin{aligned}
\langle (\phi_1)_x, \phi_2 \rangle_0 = \int_\R (\phi_1)_x \phi_2 \, dx = - \int_\R \phi_1 (\phi_2)_x \, dx &\leq \| \phi_1 \|_0 \| \pd_x \phi_2 \|_0 \nonumber \\
&\leq C(\e_1) \| \phi_1 \|_0^2 + C \e_1 \| \pd_x \phi_1 \|_0^2. 
%\label{adjoint_2}
\end{aligned}
\]
Furthermore, using integration by parts one arrives at
\[
\begin{aligned}
- \langle (\bbet \phi_2)_x, \phi_2 \rangle_0 &= - \int_\R (\bbet \phi_2)_x \phi_2 \, dx = \int_\R \bbet \phi_2 (\phi_2)_x \, dx
= \frac{1}{2} \int_\R \bbet (\phi_2^2)_x \, dx \nonumber\\
&= - \frac{1}{2} \int_\R \bbet_x (\phi_2)^2 \, dx \leq \frac{1}{2} \| \bbet_x \|_{L^\infty} \| \phi_2 \|_0^2 \leq C \| \phi_2 \|_0^2.
%\label{adjoint_3}
\end{aligned}
\]
Consequently, we have 
\[
\mu \langle (\phi_2)_{xx}, \phi_2 \rangle_0 = \mu \int_\R (\phi_2)_{xx} \phi_2 \, dx = -\mu \int_\R (\pd_x \phi_2)^2 \, dx
= -\mu \| \pd_x \phi_2 \|_0^2.
\]
Likewise, there holds the estimate
\[
\langle g, \phi_2 \rangle_0 \leq \| g \|_0 \| \phi_2 \|_0 \leq C \| g \|_0^2 + C \| \phi_2 \|_0^2.
\]
Combining these estimates, we obtain
\[
- \tfrac{1}{2} \pd_t (\| \phi_2 \|_0^2) + \left ( \mu - C \e_1 \right ) \| \pd_x \phi_2 \|_0^2
\leq C(\e_1) \left ( \| \phi \|_0^2 + \| g_2 \|_0^2 \right ). 
%\label{ineq:adjoint}
\]
Choosing $\e_1 >0 $ sufficiently small, we deduce
\begin{equation}
- \tfrac{1}{2} \pd_t (\| \phi_2 \|_0^2) \leq C \left ( \| \phi \|_0^2 + \| g_2 \|_0^2 \right ). \label{ineq_adjoint}
\end{equation}
Now, let us take the scalar product of \eqref{eq:second_component_adjoint_sys} with $-(\phi_2)_{xx}$. This yields
\begin{equation*}
 \langle \pd_t \phi_2, (\phi_2)_{xx} \rangle_0 = -\langle (\phi_2)_x, (\phi_2)_{xx} \rangle_0 + \langle (\bbet \phi_2)_x, (\phi_2)_{xx} \rangle_0 - \mu \langle (\phi_2)_{xx}, (\phi_2)_{xx} \rangle_0 - \langle g, (\phi_2)_{xx} \rangle_0.
\end{equation*}
Integrate by parts in order to get
\begin{align}
\langle \pd_t \phi_2, (\phi_2)_{xx} \rangle_0 &= \int_\R (\pd_t \phi_2) (\phi_2)_{xx} \, dx
= - \int_\R (\phi_2)_x \pd_t (\phi_2)_x \, dx = - \tfrac{1}{2} \pd_t (\| \pd_x \phi_2 \|_0^2) \label{eq:adj_1}.
\end{align}
Then, we infer
\begin{equation}
\label{eq:adj_2}
- \langle (\phi_2)_{xx}, (\phi_1)_x \rangle_0 = - \int_\R (\phi_2)_{xx} (\phi_1)_x \, dx = \int_\R \phi_1 (\phi_2)_{xxx} \, dx.
\end{equation}
Moreover, we obtain
\begin{align}
\langle (\bbet \phi_2)_x, (\phi_2)_{xx} \rangle_0 &= \int_\R (\bbet \phi_2)_x (\phi_2)_{xx} \, dx 
= \int_\R (\bbet_x \phi_2 + \bbet (\phi_2)_x ) (\phi_2)_{xx} \, dx \nonumber\\
&= \int_\R \bbet_x \phi_2 (\phi_2)_{xx} \, dx + \int_\R \bbet (\phi_2)_x (\phi_2)_{xx} \, dx \nonumber\\
&= -\int_\R (\bbet_x \phi_2)_x (\phi_2)_x \, dx + \tfrac{1}{2} \int_\R \bbet \left ( (\pd_x \phi_2)^2 \right )_x \, dx \nonumber\\
&= -\int_\R (\bbet_{xx} \phi_2 + \bbet_x (\phi_2)_x) (\phi_2)_x \, dx - \tfrac{1}{2} \int_\R \bbet_x  (\pd_x \phi_2)^2 \, dx \nonumber\\
&= -\int_\R \bbet_{xx} \phi_2 (\phi_2)_x \, dx - \tfrac{3}{2} \int_\R \bbet_x ((\pd_x \phi_2))^2 \, dx \nonumber\\
&\leq \| \bbet_{xx} \|_{L^{\infty}} \| \phi_2 \|_0 \| \pd_x \phi_2 \|_0 + \tfrac{3}{2} \| \bbet_x \|_{L^{\infty}} \| \pd_x \phi_2 \|_0^2 \nonumber\\
&\leq C \left ( \| \phi_2 \|_0^2 + \| \pd_x \phi_2 \|_0^2 \right ). \label{adj_3}
\end{align}
This yields,
\begin{equation}
\label{adj_4}
-\mu \langle (\phi_2)_{xx}, (\phi_2)_{xx} \rangle_0 = -\mu \| \pd_x^2 \phi_2 \|_0^2.
\end{equation}
In addition, one can estimate
\begin{equation}
-\langle g_2, (\phi_2)_{xx} \rangle_0 = \int_\R g_2 (\phi_2)_{xx} \, dx \leq \| g_2 \|_0 \| \pd_x^2 \phi_2 \|_0
\leq C(\e_2) \| g_2 \|_0^2 + C \e_2 \| \pd_x^2 \phi_2 \|_0^2.
\label{adj_5}
\end{equation}
Using \eqref{eq:adj_1}, \eqref{eq:adj_2}, \eqref{adj_3}, \eqref{adj_4}, and \eqref{adj_5} we deduce
\begin{align}
-\tfrac{1}{2} \pd_t (\| \pd_x \phi_2 \|_0^2) + \left ( \mu - C \e_2 \right )& \| \pd_x^2 \phi_2 \|_0^2
- \int_\R \phi_1 (\phi_2)_{xxx} \, dx \nonumber\\
&\leq C(\e_2) \left (  \| \phi_2 \|_0^2 + \| \pd_x \phi_2 \|_0^2 + \| g_2 \|_0^2 \right ). \label{ineq_adj}
\end{align}
The first component of \eqref{eq:adjoint_system} implies that
\begin{equation*}
(\phi_2)_{xxx} = \frac{2}{k^2} \pd_t \phi_1 - \frac{2}{k^2} (\balp \phi_2)_x - \frac{2}{k^2} ( \zeta \phi_2 )_{xx}
- \frac{2}{k^2} ( \eta  \phi_2 )_x + \frac{2}{k^2} g_1.
\end{equation*}
Therefore, we have
\[
\begin{aligned}
- \int_\R \phi_1 (\phi_2)_{xxx} \, dx &= \left \langle - \frac{2}{k^2} \pd_t \phi_1 + \frac{2}{k^2} (\balp \phi_2)_x +
 \frac{2}{k^2} ( \zeta \phi_2 )_{xx}
+ \frac{2}{k^2} ( \eta \phi_2 )_x - \frac{2}{k^2} g_1, \phi_1 \right \rangle_0. 
%\label{eq:adj_6}
\end{aligned}
\]
This yields,
\[
-\frac{2}{k^2} \langle \pd_t \phi_1, \phi_1 \rangle_0 = - \frac{1}{k^2} \pd_t \| \phi_1 \|_0^2.
\]
Moreover,
\[
\begin{aligned}
\langle (\balp \phi_2)_x, \phi_1 \rangle_0 &= \int_\R \phi_1 (\balp \phi_2)_x \, dx
= \int_\R \balp_x \phi_1 \phi_2 \, dx + \int_\R \balp \phi_1 (\phi_2)_x \, dx \nonumber\\
&\leq \| \balp \|_{L^\infty} \| \phi_1 \|_0 \| \pd_x \phi_2 \|_0 + \| \balp_x \|_{L^\infty} \| \phi_1 \|_0 \| \phi_2 \|_0 \nonumber \\
&\leq C \| \phi \|_0^2 + C \| \pd_x \phi_2 \|_0^2. 
%\label{adj_8}
\end{aligned}
\]
We also have,
\[
\begin{aligned}
\left \langle \frac{2}{k^2} ( \zeta \phi_2 )_{xx}, \phi_1 \right \rangle_0
&= \langle (\zeta \phi_2)_{xx}, \phi_1 \rangle_0 \\
&= \frac{2}{k^2} \left ( \int_\R \zeta_{xx} \phi_2 \phi_1 \, dx + 2 \int_\R \zeta_x (\phi_2)_x \phi_1 \, dx + \int_\R \zeta (\phi_2)_{xx} \phi_1 \, dx \right ) \\
&\leq \frac{2}{k^2} \Big( \| \zeta_{xx} \|_{L^\infty} \| \phi_2\|_0 \| \phi_1 \|_0
+ 2 \| \zeta_x \|_{L^{\infty}} \| \pd_x \phi_2\|_0 \| \phi_1\|_0
+ \| \zeta \|_{L^{\infty}} \| \pd_x^2 \phi_2 \|_0 \| \phi_1 \|_0 \Big) \\
&\leq C(\e_3) \| \phi \|_0^2 + C \| \pd_x \phi_2 \|_0^2 + C \e_3 \| \pd_x^2 \phi_2 \|_0^2. 
%\label{adj_9}
\end{aligned}
\]
Therefore,
\[
\begin{aligned}
\frac{2}{k^2} \left \langle (\eta \phi_2 )_x, \phi_1 \right \rangle_0 &= \frac{2}{k^2} \left ( \int_\R \eta_x \phi_2 \phi_1 \, dx  + \int_\R \eta (\phi_2)_x \phi_1 \, dx \right ) \nonumber \\
&\leq \frac{2}{k^2} \Big( \| \eta_x \|_{L^{\infty}} \| \phi_1\|_0 \| \phi_2\|_0
+ \| \eta \|_{L^{\infty}} \| \phi_1 \|_0 \| \pd_x \phi_2 \|_0 \Big)  \\
&\leq C \big( \| \phi \|_0^2 + \| \pd_x \phi_2 \|_0^2 \big). 
%\label{adj_10}
\end{aligned}
\]
Also, we clearly have
\begin{equation}
\label{adj_11}
-\frac{2}{k^2} \langle g_1, \phi_1 \rangle_0 \leq \frac{2}{k^2} \| g_1 \|_0 \| \phi_1 \|_0 \leq C \big( \| g_1 \|_0^2 +  \| \phi_1 \|_0^2 \big).
\end{equation}
Let $c \in (0,1)$. Thanks to estimates \eqref{ineq_adj} thru \eqref{adj_11}, choosing $\e_2, \e_3 > 0$ sufficiently small we obtain
\[
\begin{aligned}
- \pd_t \left (\frac{1}{k^2}  \| \phi_1 \|_0^2 + \frac{1}{2} \| \pd_x \phi_2 \|_0^2 \right ) + c \mu \| \pd_x^2 \phi_2 \|_0^2 \leq C \left ( \| \phi \|_0^2 + \| \pd_x \phi_2 \|_0^2 + \| g \|_0^2 \right ). 
%\label{ineq_adjoint_2}
\end{aligned}
\]
Since $c \mu \| \pd_x^2 \phi_2 \|^2 \geq 0$, last inequality implies that
\begin{equation}
- \pd_t \left (\frac{1}{k^2}  \| \phi_1 \|_0^2 + \frac{1}{2} \| \pd_x \phi_2 \|_0^2 \right )
\leq C \left ( \| \phi \|_0^2 + \| \pd_x \phi_2 \|_0^2 + \| g \|_0^2 \right ). \label{ineq_adjoint_3}
\end{equation}
Multiplying \eqref{ineq_adjoint} by $2/k^2$ and adding it to \eqref{ineq_adjoint_3} we arrive at
\begin{equation*}
-\pd_t \left ( \frac{1}{k^2} \| \phi \|_0^2 + \frac{1}{2} \| \pd_x \phi_2 \|_0^2 \right ) 
\leq C \left ( \frac{1}{k^2} \| \phi \|_0^2 + \frac{1}{2} \| \pd_x \phi_2 \|_0^2 + \| g \|_0^2 \right ).
\end{equation*}
Let us change the variable $\tau = T - t$, $\phi(t) = \tilde{\phi}(\tau)$. Then, the initial condition of \eqref{eq:adjoint_system} implies that 
$\tilde{\phi}(0) = 0$. Applying Gronwall inequality to the resulting relation we obtain
\eqref{estimate_adjoint_system}. The lemma is now proved.
\end{proof}

\subsubsection{Negative norm estimates}

First, let us introduce some definitions. We denote the Fourier transform operator by $\cF$ and $\hat{u}(\xi) = (\cF u)(\xi)$. For $\xi \in \R$, we introduce the notation $\brc = \sqrt{ 1+ \xi^2}$. Let $u:\R \rightarrow \R$. For $s \in \R$ we define the operator $\La^s u = \cF^{-1}(\brc^s \cF u)$. For $n \in \N_0$ we have $\pd_x^n \La^s u = \La^s \pd_x^n u$. Denote by $\cS$ the space of Schwartz functions on $\R$. For a function $v : \R \rightarrow \R^2$ we set 
$\La^s v = ( \La^s v_1, \La^s v_2 )^\top$. 
%Then, in this section we define the Sobolev space for $s \in \R$
%\begin{equation*}
%H^s(\R) = \{ f \in \cS' : \La^s f \in L^2(\R) \}.
%\end{equation*}
%The space $H^s(\R)$ is equipped with the norm $\| f \|_s = \| \La^s f \|_0$. For $s \in \N_0$ the norm $\| \cdot \|_s$ is equivalent to the corresponding norm defined in Section \ref{sec:preliminaries}. 
Let us define the space 
\[
X := \bigcap_{n \in \N_0} H^n(\R).
\]
For any $s \in \R$ we have
\begin{equation}
\| u \|_{s+1}^2 = \| u \|_{s}^2 + \| \pd_x \La^s u \|_0^2. \label{property_norm}
\end{equation}
The following lemma will be used to estimate contributions coming from commutators:
\begin{lemma}
\label{lemma:commutator_estimate}
Let $s \in \R$ and $u \in H^{|s-1|+2}(\R)$. Then, there exists a constant $C = C_s$, such that
\begin{equation*}
\| [\La^s, u] f \| \leq C \| u \|_{|s-1|+2} \| f \|_{s-1},
\end{equation*}
for $f \in X$.
\end{lemma}
\begin{proof}
The proof follows from Lemma 6.16, p. 202 in \cite{Fo95} by a density argument.
\end{proof}
Suppose $c \in \R$ is a constant and $f, u : \R \rightarrow \R$. Then,
\begin{equation}
[\La^s, u + c]f = [\La^s, u]f.\label{equality_commutator}
\end{equation}

We have the following result.
\begin{lemma}
\label{lemma:coefficients_Sobolev_space}
Suppose $\bw, g \in C_0^\infty([0,T] \times \R)$ and
\begin{equation*}
\sup_{x \in \R, \, t \in [0,T]} \frac{1}{\br(x,t) + \rc} \leq C_0.
\end{equation*}
Then,
\begin{equation*}
\pd_x^n \left( \balp  - \left( \Big( \frac{\mc}{\rc} \Big)^2 + p'(\rc) \right) \right ) ,\mbox{ }
\pd_x^n \Big( \bbet + \frac{2 \mc}{\rc} \Big),\mbox{ }\pd_x^n \zeta,\mbox{ }\pd_x^n \eta \in L^2(\R),\quad \forall \, n \in \N_0.
\end{equation*}
\end{lemma}
\begin{proof}
We clearly have
\begin{align*}
\left (\frac{\brm + \mc}{\br + \rc} \right )^2 - \left ( \frac{\mc}{\rc} \right ) = \frac{(\rc)^2 \brm^2 - (\mc)^2 \br^2 + 2 \rc \mc (\rc \brm - \mc \br)}{(\rc)^2 (\br + \rc)^2}
\in L^2(\R).
\end{align*}
Let $a = \min\{ \rc, 1/C_0 \}$, $b = \|\br\|_{L^\infty} + \rc$, and
\begin{equation*}
G(z) = \int_0^1 p''(t z + \rc) dt,\mbox{ }z > -\rc.
\end{equation*}
Then, there holds
\begin{equation*}
p'(\br(x) + \rc) - p'(\rc) = G(\br(x))\br(x), \qquad x \in \R.
\end{equation*}
Moreover,
\begin{equation*}
\|G(\rho(\cdot))\|_{L^{\infty}} \leq \sup_{z \in [a,b]} | p''(z)|.
\end{equation*}
Therefore,
\begin{align*}
\| p'(\br(\cdot) + \rc) - p'(\rc) \|_0 = \| G(\br(x))\br(x) \|_0 \leq \|G(\rho(\cdot))\|_{L^{\infty}} \| \br \|_0 \leq \sup_{z \in [a,b]} | p''(z)| \| \br \|_0.
\end{align*}
%%% AQUIIII
Hence $p'(\br(\cdot) + \rc) - p'(\rc) \in L^2(\R)$, and consequently
\begin{equation*}
\balp - \Big( \Big( \frac{\mc}{\rc} \Big)^2 + p'(\rc) \Big) \in L^2(\R).
\end{equation*}
Similarly,
\begin{equation*}
-\frac{2 (\brm + \mc)}{\br + \rc} + \frac{2 \mc}{\rc} = 2 \frac{\mc \br - \rc \brm}{\rc (\br + \rc)}.
\end{equation*}
Therefore,
\begin{equation*}
\bbet + \frac{2 \mc}{\rc} \in L^2(\R).
\end{equation*}
In addition, we have $\pd_x^n \balp$, $\pd_x^n \bbet \in L^2(\R)$, $n \in \N$.
%\begin{equation*}
%\pd_x^n \alpha,\mbox{ }\pd_x^n \beta \in L^2(\R),\mbox{ }n \in \N.
%\end{equation*}
Since $\zeta$ and $\eta$ are a product of a coefficient and $\br_x$, we have
\begin{equation*}
\pd_x^n \zeta,\mbox{ }\pd_x^n \eta \in L^2(\R),\mbox{ }n \in \N_0,
\end{equation*}
and the proof is complete.
\end{proof}

\begin{lemma}
Suppose $s \in \R$, $\bw, g \in C_0^\infty([0,T] \times \R)$ and
\begin{equation*}
\sup_{x \in \R, \, t \in [0,T]} \frac{1}{\br(x,t) + \rc} \leq C_0.
\end{equation*}
Then,
\begin{equation}
\| \La^s \phi(t) \|_0^2 + \| \pd_x \La^s \phi_2(t) \|_0^2 \leq C(T) \int_{0}^{T} \| \La^s g(\tau) \|_0^2 d\tau,
\label{ineq_negative_order}
\end{equation}
for all $t \in [0,T]$, where $C(T)$ depends only on $\bw$, $T$ and the parameters of \eqref{QHD}.
\end{lemma}
\begin{proof}
Let $s \in \R$. Applying the operator $\La^s$ to \eqref{eq:adjoint_system} we obtain that the equation satisfied by $\La^s \phi$ is
\begin{equation}
\label{adj_s_1}
-\pd_t \La^s \phi = \cL^* \La^s \phi + [ \La^s, \cL^*] \phi - \La^s g.
\end{equation}
Take the scalar product of the second component of \eqref{adj_s_1} with $\La^s \phi_2$. The result is
\[
-\langle \pd_t \La^s \phi_2, \La^s \phi_2 \rangle_0 = \langle \cL_2^* \La^s \phi, \La^s \phi_2 \rangle_0
+ \langle [ \La^s, \cL_2^*] \phi, \La^s \phi_2 \rangle_0 - \langle \La^s g_2, \La^s \phi_2 \rangle_0.
\]
Then, we have,
\begin{equation}
-\langle \pd_t \La^s \phi_2, \La^s \phi_2 \rangle_0 = -\int_\R (\pd_t \La^s \phi_2)(\La^s \phi_2) \, dx
= -\tfrac{1}{2} \pd_t (\| \phi_2 \|_{s}^2). \label{eq:adj_s_3}
\end{equation}
Moreover,
\begin{align*}
\langle \cL_2^* \La^s \phi, \La^s \phi_i \rangle_0
= \left \langle \pd_x \La^s \phi_1 - \pd_x (\bbet \La^s \phi_2) + \mu \pd_x^2  \La^s \phi_2, \La^s \phi_2 \right \rangle_0.
\end{align*}
Then, after integration by parts and by \eqref{property_norm}, we arrive at
\[
\begin{aligned}
\langle \pd_x \La^s \phi_1, \La^s \phi_2 \rangle_0 = \int_\R (\pd_x \La^s \phi_1) ( \La^s \phi_2) \, dx &= - \int_\R (\La^s \phi_1) (\pd_x \La^s \phi_2) \, dx \nonumber\\
&\leq \| \La^s \phi_1 \|_0 \| \pd_x \La^s \phi_2 \|_0 \nonumber\\
&\leq \| \phi_1 \|_s \| \phi_2\|_{s+1} \nonumber\\
&\leq C(\e_1) \| \phi_1 \|_{s}^2 + C \e_1 \| \phi_2 \|_{s+1}^2 \nonumber\\
&= C(\e_1) \| \phi_1 \|_{s}^2 + C \e_1 \| \phi_2 \|_{s}^2 + C \e_1 \| \pd_x \La^s \phi_2\|_0^2. 
%\label{adj_s_4}
\end{aligned}
\]
In the same fashion, we estimate
\[
\begin{aligned}
-\langle \pd_x (\bbet \La^s \phi_2), \La^s \phi_2 \rangle_0
&= - \int_\R \pd_x (\bbet \La^s \phi_2) (\La^s \phi_2) \, dx \nonumber\\
&= \int_\R \bbet (\La^s \phi_2) (\pd_x \La^s \phi_2) \, dx \nonumber\\
&\leq \| \bbet \|_{L^{\infty}} \| \La^s \phi_2 \|_0 \| \pd_x \La^s \phi_2 \|_0 \nonumber\\
&\leq C \| \phi_2 \|_s \| \phi_2\|_{s+1} \nonumber\\
&\leq C(\e_2) \| \phi_2 \|_s^2 + C \e_2 \| \phi_2\|_{s+1}^2 \nonumber\\
&\leq C(\e_2) \| \phi_2 \|_{s}^2 + C \e_2 \| \phi_2\|_{s}^2 + C \e_2 \| \pd_x \La^s \phi_2 \|_0^2. 
%\label{adj_s_5}
\end{aligned}
\]
Furthermore,
\begin{align*}
\langle \mu \pd_x^2  \La^s \phi_2, \La^s \phi_2 \rangle_0
= \mu \int_\R (\pd_x^2 \La^s \phi_2) (\La^s \phi_2) \, dx = -\mu \int_\R (\pd_x \La^s \phi_2)^2 \, dx = - \mu \| \pd_x \La^s \phi_2 \|_0^2.
\end{align*}
Then,
\begin{equation}
- \langle \La^s g_2, \La^s \phi_2 \rangle_0 \leq \| \La^s g_2 \|_0 \| \La^s \phi_2 \|_0 
= \| g_2 \|_s \| \phi_2 \|_s \leq C \big( \| g_2 \|_s^2 +  \| \phi_2 \|_s^2\big). \label{adj_s_6}
\end{equation}
Now, let us estimate the contribution coming from the commutator. We have
\begin{equation*}
[\La^s, \cL_2^* ] \phi = a_1 + a_2,
\end{equation*}
where
\begin{equation*}
a_1 = -[\La^s, \bbet]\pd_x \phi_2,\mbox{ } a_2 = - [\La^s, \bbet_x] \phi_2.
\end{equation*}
Then, in view of \eqref{equality_commutator}, we obtain
\begin{align*}
a_1 = - \left [\La^s, \bbet + \frac{2 \mc}{\rc} \right ]\pd_x \phi_2.
\end{align*}
Hence, using Lemmata \ref{lemma:commutator_estimate} and \ref{lemma:coefficients_Sobolev_space}, we infer
\begin{align}
\| [\La^s, \cL_2^* ] \phi \|_0 &\leq \| a_1 \|_0 + \| a_2 \|_0 \nonumber\\
&= \left \| \left [\La^s, \bbet + \frac{2 \mc}{\rc} \right ]\pd_x \phi_2 \right \|_0
+ \left \| [\La^s, \bbet_x] \phi_2 \right \|_0 \nonumber \\
&\leq C \left \| \bbet + \frac{2 \mc}{\rc} \right \|_{|s-1|+2} \| \pd_x \phi_2 \|_{s-1}
+ C \| \bbet_x \|_{|s-1|+2} \| \phi_2 \|_{s-1} \nonumber \\
&\leq C \big( \| \phi_2\|_{s} +  \| \phi_2\|_{s-1} \big)\nonumber \\
&\leq C \| \phi_2 \|_{s}. \label{adj_s_7}
\end{align}
Henceforth,
\begin{equation}
\langle [ \La^s, \cL_2^*] \phi, \La^s \phi_2 \rangle_0
\leq \| [\La^s, \cL_2^* ] \phi \|_0  \| \La^s \phi_2\|_0 \leq C \| \phi_2 \|_s^2. \label{adj_s_8}
\end{equation}

Using estimates \eqref{eq:adj_s_3} thru \eqref{adj_s_6}, \eqref{adj_s_8} and choosing $\e_1, \e_2 > 0$ sufficiently small, we deduce
\begin{equation}
-\tfrac{1}{2}\pd_t ( \| \phi_2\|_{s}^2 ) \leq C \left ( \| \phi \|_{s}^2 + \| g_2 \|_{s}^2 \right ). \label{adj_s_9}
\end{equation}
Now, taking the scalar product of the second component of \eqref{adj_s_1} with $-\pd_x^2 \La^s \phi_2$ we obtain
\begin{align*}
\langle \pd_t \La^s \phi_2, \pd_x^2 \La^s \phi_2 \rangle_0 = -\langle \cL_2^* \La^s \phi, \pd_x^2 \La^s \phi_2\rangle_0 - \langle [ \La^s, \cL_2^*] \phi, \pd_x^2 \La^s \phi_2 \rangle_0 + \langle \La^s g_2, \La^s \pd_x^2 \phi_2 \rangle_0.
\end{align*}
This yields,
\begin{equation}
\langle \pd_t \La^s \phi_2, \pd_x^2 \La^s \phi_2 \rangle_0  = -\tfrac{1}{2}\pd_t \| \pd_x \La^s \phi_2 \|_0^2. \label{adj_s_pd2_1}
\end{equation}
Moreover,
\begin{equation}
- \langle \pd_x \La^s \phi_1, \pd_x^2 \La^s \phi_2 \rangle_0 
= - \int_\R (\pd_x \La^s \phi_1) (\pd_x^2 \La^s \phi_2) \, dx
= \int_\R (\La^s \phi_1) (\pd_x^3 \La^s \phi_2) \, dx. \label{adj_s_pd2_2}
\end{equation}
Furthermore,
\[
\begin{aligned}
\langle \pd_x (\bbet \La^s \phi_2), \pd_x^2 \La^s \phi_2 \rangle_0
 &= \int_\R \pd_x (\bbet \La^s \phi_2) \pd_x^2 \La^s \phi_2 \, dx \nonumber\\
&= \int_\R (\bbet_x \La^s \phi_2) (\pd_x^2 \La^s \phi_2)\, dx
 + \int_\R \bbet (\pd_x \La^s \phi_2) (\pd_x^2 \La^s \phi_2)\, dx \nonumber\\
& = -\int_\R \pd_x (\bbet_x \La^s \phi_2)(\pd_x \La^s \phi_2)\, dx
+ \frac{1}{2} \int_\R \bbet \pd_x \left ( (\pd_x \La^s \phi_2 )^2 \right ) \, dx \nonumber\\
&= -\int_\R \bbet_{xx} (\La^s \phi_2) \pd_x \La^s \phi_2 \, dx
- \int_\R \bbet_x (\pd_x \La^s \phi_2)^2 \, dx - \frac{1}{2} \int_\R \bbet_x (\pd_x \La^s \phi_2 )^2 \, dx \nonumber\\
&\leq \| \bbet_{xx} \|_{L^{\infty}} \| \La^s \phi_2 \|_0 \| \pd_x \La^s \phi_2 \|_0
+ \frac{3}{2}\| \bbet_x \|_{L^{\infty}} \| \pd_x \La^s \phi_2 \|_0^2 \nonumber\\ 
&\leq C \|\phi_2 \|_{s+1}^2. 
%\label{adj_s_pd2_3}
\end{aligned}
\]
%Then,
%\begin{equation}
%- \mu \langle \pd_x^2 \La^s \phi_2, \pd_x^2 \La^s \phi_2 \rangle = -\mu \| \pd_x^2 \La^s \phi_2 \|^2.
%\label{adj_s_pd2_4}
%\end{equation}
Moreover,
\[
\langle \La^s g_2, \pd_x^2 \La^s \phi_2  \rangle_0 \leq \| \La^s g_2 \|_0 \| \pd_x^2 \La^s \phi_2 \|_0
\leq C(\e_3) \| g_2 \|_{s}^2  + C \e_3 \| \pd_x^2 \La^s \phi_2 \|_0^2.
%\label{adj_s_pd2_5}
\]
Using \eqref{adj_s_7} we obtain
\begin{align}
-\langle [ \La^s, \cL_2^* ] \phi, \pd_x^2 \La^s \phi_2 \rangle_0 &\leq \| [ \La^s, \cL_2^* ] \phi \|_0 \| \pd_x^2 \La^s \phi_2 \|_0 \nonumber \\
&\leq C \| \phi_2 \|_s \| \pd_x^2 \La^s \phi_2 \|_0 \nonumber \\
&\leq C(\e_4) \| \phi_2 \|_{s}^2 + C \e_4 \| \pd_x^2 \La^s \phi_2 \|_0^2.
\label{adj_s_pd2_6}
\end{align}
Let $c \in (0,1)$. Using estimates \eqref{adj_s_pd2_1} thru
\eqref{adj_s_pd2_6} and choosing $\e_3, \e_4 > 0$ sufficiently small we deduce
\begin{align}
-\frac{1}{2} \pd_t \| \pd_x \La^s \phi_2 \|_0^2 - \int_\R (\La^s \phi_1) ( \pd_x^3 \La^s \phi_2)\, dx
+ c \mu \| \pd_x^2 \La^s \phi_2 \|_0^2 \leq C \left ( \| \phi_2 \|_{s+1}^2 + \| g_2 \|_{s}^2 \right ). \label{adj_s_pd2_7}
\end{align}
The first component of \eqref{adj_s_1} is
\begin{equation}
-\pd_t \La^s \phi_1 = \cL_1^* \La^s \phi + [\La^s , \cL_1^*] \phi - \La^s g_1, \label{adj_s_firrst_eq_1}
\end{equation}
where
\begin{equation*}
\cL_1^* \La^s \phi = - \pd_x (\balp \La^s \phi_2) - \frac{k^2}{2} \pd_x^3 \La^s \phi_2 - \pd_x^2( \zeta \La^s \phi_2 )
-   \pd_x (  \eta \La^s \phi_2 ).
\end{equation*}
Let us express $\La^s \phi_2$ from \eqref{adj_s_firrst_eq_1}. The result is
\[
\begin{aligned}
\pd_x^3 \La^s \phi_2 &= \frac{2}{k^2}\pd_t \La^s \phi_1 - \frac{2}{k^2}\pd_x (\balp \La^s \phi_2) - \frac{2}{k^2} \pd_x^2 
( \zeta \La^s \phi_2 ) - \frac{2}{k^2} \pd_x ( \eta \La^s \phi_2 )
+ \\
&\quad + \frac{2}{k^2} [\La^s , \cL_1^*] \phi - \frac{2}{k^2}\La^s g_1. 
\end{aligned}
\]
Let us substitute last equation into the second term of \eqref{adj_s_pd2_7}. This yields
\begin{equation}
- \frac{2}{k^2} \int_\R (\La^s \phi_1) (\pd_t \La^s \phi_1) \, dx
= -\frac{1}{k^2} \pd_t (\| \La^s \phi_1 \|_0^2). \label{adj_s_first_eq_3}
\end{equation}
Then,
\[
\begin{aligned}
\frac{2}{k^2} \int_\R (\La^s \phi_1) (\balp \La^s \phi_2)_x \, dx &= \frac{2}{k^2} \int_\R (\La^s \phi_1) (\balp_x \La^s \phi_2) \, dx
+ \frac{2}{k^2} \int_\R (\La^s \phi_1) (\balp \pd_x \La^s \phi_2) \, dx \nonumber \\
&\leq \frac{2}{k^2} \| \balp_x \|_{L^{\infty}} \| \La^s \phi_1 \|_0 \| \La^s \phi_2 \|_0
+ \frac{2}{k^2} \| \balp \|_{L^{\infty}} \| \La^s \phi_1 \|_0 \| \pd_x \La^s \phi_2 \|_0 \nonumber\\
&\leq C \big( \| \phi_1\|_{s}^2 + \| \phi_2\|_{s+1}^2\big). 
%\label{adj_s_first_eq_4}
\end{aligned}
\]
Furthermore,
\[
\begin{aligned}
\frac{2}{k^2} \int_\R (\La^s \phi_1) &\pd_x^2 ( \zeta \La^s \phi_2 ) \, dx = -\frac{2}{k^2} \int_\R (\La^s \phi_1) \pd_x^2 \left (\zeta \La^s \phi_2 \right ) \, dx \nonumber\\
&= -\frac{2}{k^2} \int_\R (\La^s \phi_1) \left ( \zeta_{xx} \La^s \phi_1 + 2 \zeta_x \pd_x \La^s \phi_2
+ \zeta \pd_x^2 \La^s \phi_2 \right ) \, dx \nonumber\\
&\leq \frac{2}{k^2} \Big( \| \zeta_{xx} \|_{L^{\infty}} \| \La^s \phi_1 \|_0^2  
+ 2 \| \zeta_x \|_{L^{\infty}} \| \La^s \phi_1 \|_0 \| \pd_x \La^s \phi_2 \|_0 + \| \zeta \|_{L^{\infty}} \| \La^s \phi_1 \|_0 \| \pd_x^2 \La^s \phi_2 \|_0 \Big) \nonumber \\
&\leq C \big( \| \phi_1 \|_{s}^2 +  \| \phi_1 \|_{s} \| \phi_2 \|_{s+1}  + \| \phi_1 \|_{s} \| \pd_x^2 \La^s \phi_2  \|_0 \big) \nonumber \\
&\leq C(\e_5) \| \phi_1 \|_{s}^2 + C \| \phi_2\|_{s+1}^2 + C \e_5 \| \pd_x^2 \La^s \phi_2 \|_0^2. 
%\label{adj_s_first_eq_5}
\end{aligned}
\]
Now, let us estimate
\[
\begin{aligned}
\frac{2}{k^2} \int_\R (\La^s \phi_1) \pd_x ( \eta \La^s \phi_2) \, dx &= \frac{2}{k^2} \int_\R (\La^s \phi_1) \left ( \eta_x \La^s \phi_2 + \eta \pd_x \La^s \phi_2 \right ) \, dx \nonumber\\
&\leq \frac{2}{k^2} \| \eta_x \|_{L^{\infty}} \| \La^s \phi_1 \|_0 \| \La^s \phi_2 \|_0
+ \frac{2}{k^2} \| \eta \|_{L^{\infty}} \| \La^s \phi_1 \|_0 \| \pd_x \La^s \phi_2 \|_0 \nonumber\\
&\leq C \big( \| \phi \|_{s}^2 + \| \phi_2 \|_{s+1}^2 \big). 
%\label{adj_s_first_eq_6}
\end{aligned}
\]
Moreover,
\[
\frac{2}{k^2} \int_\R (\La^s \phi_1) (\La^s g_1) \, dx \leq \frac{2}{k^2} \| \La^s \phi_1 \|_0 \| \La^s g_1 \|_0
\leq C \big(\|\phi_1 \|_{s}^2 + \| g_1 \|_{s}^2 \big). 
%\label{adj_s_first_eq_7}
\]
Now, let us estimate the contribution coming from the commutator. Similarly to \eqref{adj_s_7}, using Lemmata \ref{lemma:commutator_estimate} and \eqref{equality_commutator}, we obtain
\begin{align}
 \| [\La^s, \cL_1^*] \phi \| &\leq  \|  [ \La^s, \balp_x] \phi_2 \|_0 
+ \left \| \left [\La^s,\balp 
- \Big( \Big( \frac{\mc}{\rc} \Big)^2 + p'(\rc) \Big) \right ] \pd_x \phi_2 \right \|_0 \nonumber \\
&+ \| [ \La^s, \zeta_{xx}]\phi_2 \|_0 + 2 \| [ \La^s, \zeta_x] \pd_x \phi_2 \|_0 \nonumber \\
&+ \| [\La^s, \zeta] \pd_x^2 \phi_2 \|_0 + \| [ \La^s, \eta_x ] \phi_2 \|_0 + \| [ \La^s, \eta] \pd_x \phi_2 \|_0 \nonumber\\
&\leq C \| \phi_2 \|_{s+1}. \label{adj_s_first_eq_8}
\end{align}
Using estimates \eqref{adj_s_first_eq_3} thru \eqref{adj_s_first_eq_8}, and choosing $\e_5 > 0$ sufficiently small,  we arrive at
\[
-\pd_t \left ( \frac{1}{k^2}  \| \La^s \phi_1 \|_0^2 + \frac{1}{2}  \| \pd_x \La^s \phi_2 \|_0^2 \right )
\leq C \left ( \| \phi_1 \|_{s}^2 + \| \phi_2 \|_{s+1}^2 + \| g \|_{s}^2 \right ). 
%label{adj_s_first_eq_9}
\]
Multiplying \eqref{adj_s_9} by $2/k^2$ and adding it to the last inequality we infer
\begin{align}
- \pd_t \left ( \frac{1}{k^2} \| \La^s \phi \|_0^2 + \frac{1}{2} \| \pd_x \La^s \phi_2 \|_0^2 \right ) \leq C \left (  \frac{1}{k^2} \| \La^s \phi \|_0^2 + \frac{1}{2} \| \pd_x \La^s \phi_2 \|_0^2 + \| \La^s g \|_0^2 \right )
\label{res_ineq_s}
\end{align}
Applying Gronwall's inequality to \eqref{res_ineq_s} similarly as in the proof of Lemma \ref{theorem_zero_order}, we obtain \eqref{ineq_negative_order}. This yields the result.
\end{proof}
Now, we are going to show that the linear system \eqref{linear_system} has a unique solution. Let us define the operator including the time derivative 
\begin{equation}
\tL := \pd_t - \cL. \label{operator_time_derivative}
\end{equation}
Its formal adjoint is given by $\tL^* = -\pd_t - \cL^*$ and satisfies $\langle \tL w , \phi \rangle_0 = \langle w, \tL^* \phi \rangle_0$.
%\begin{equation*}
%\langle \tL w , \phi \rangle = \langle w, \tL^* \phi \rangle.
%\end{equation*}
%%VOY HERE
\begin{lemma}[existence of solution to linear system]
\label{theorem_existence_linear_system}
Suppose $n \geq 3$ and let
\begin{align*}
&f \in L^2([0,T], H^n(\R)),\mbox{ }f_1 \in L^2([0,T], H^{n+1}(\R)),\\
&\rho_0 \in H^{n+1}(\R),\mbox{ }m_0 \in H^n(\R).
\end{align*}
Then the initial value problem \eqref{linear_system} has a unique solution which satisfies the estimate \eqref{integral_estimate_higher_order}.
\end{lemma}
\begin{proof}
Let $\phi \in C_0^{\infty}([0,T] \times \R)$. We then have
\begin{equation*}
\left | \int_0^T \langle f, \phi \rangle_0 dt \right | = \left | \int_0^T \langle \La^n f, \La^{-n} \phi \rangle_0 \, dt \right | \leq
\int_0^T \| \La^n f \|_0 \| \La^{-n} \phi \|_0 \, dt.
\end{equation*}
From the inequality \eqref{ineq_negative_order} with $s = - n$ we deduce
\begin{align*}
\left | \int_0^T \langle f, \phi \rangle_0 dt \right | &\leq C(T) \int_0^T \| \La^n f \|_0
 \left ( \int_0^T \| \La^{-n} \tL^* \phi \|_0^2 \, d\tau \right )^{1/2} \!\!dt\\
&\leq C(T) \left ( \int_0^T \| \La^n f \|_0^2 \, dt \right )^2 \left ( \int_0^T \| \La^{-n} \tL^* \phi \|_0^2 \, dt \right )^{1/2}.
\end{align*}
Hence,
\begin{equation*}
\int_0^T \langle f, \phi \rangle_0 \,  dt
\end{equation*}
defines a bounded linear functional of $\tL^* \phi$ in $L^2([0,T],H^{-n}(\R))$. Applying the Hanh-Banach extension theorem and Riesz representation theorem we obtain that there exists a unique weak solution $w \in  L^2([0,T],H^n(\R))$ such that
\begin{equation*}
\int_0^T \langle f, \phi \rangle_0 \, dt = \int_0^T \langle w, \tL^* \phi \rangle_0 \, dt,
\end{equation*}
for all $\phi \in C_0^{\infty}([0,T] \times \R)$. Therefore,
\begin{equation*}
\int_\R f \varphi \, dx = \int_\R (\tL w) \varphi \, dx,\qquad \varphi \in C_0^{\infty}(\R),
\end{equation*}
for all $t \in [0,T]$ a.e. Since $n \geq 3$ the Sobolev embedding theorem implies that the solution is classical.
\end{proof}

\subsection{Proof of Theorem \ref{themlocale}}
Now, let us consider the initial value problem for system \eqref{QHD}, with
\begin{equation}
\rho(0) = \rho_0,\mbox{ }m(0) = m_0. \label{initial_condition}
\end{equation}
First we shall prove the following lemma about local existence of solutions.
\begin{lemma}
\label{lemma_local_existence}
Suppose $s \in \mathbb{N},\mbox{ }s \geq 3$. For any initial condition $(\rho_0, m_0)$ such that $\rho_0(x) \geq \delta > 0 $ and
\begin{equation*}
\rho_0 - \rc \in H^{s+1}(\R),\mbox{ }m_0 - \mc \in H^s(\R),
\end{equation*}
where $\rc > 0$ and $\mc \in \R$ are constants, there exists $T > 0$ such that the initial value problem \eqref{QHD}-\eqref{initial_condition} has a unique solution
\begin{equation*}
\rho - \rc \in L^{\infty}([0,T],H^{s+1}(\R)), \mbox{ }m - \mc \in L^{\infty}([0,T],H^s(\R)).
\end{equation*}
Moreover, $w = (\rho - \rc, m - \mc)$ satisfies the estimate
\begin{equation*}
\vertiii{w }_{s,[0,T]}^2 \leq C_n(T) \Big( \| w_0 \|_s^2 + \| \rho_0 - \rc \|_{s+1}^2 \Big),
\end{equation*}
with $w_0 = (\rho_0 - \rc, m_0 - \mc)$, and where the positive constant $C_n(T)$ depends only on $\rc$, $\mc$, the parameters of \eqref{QHD} and $T$.
\end{lemma}
\begin{proof}
The nonlinear problem can be written as
\begin{equation}
\left\{
\begin{aligned}
\tL(w)w &= 0,\\
w(0) &= w_0.
\end{aligned} 
\right.
\label{eq:quasilinear}
\end{equation}
In order to recast the system so that we have homogeneous initial data, let $\tilde{w}_i$ be the classical solutions to the heat equation,
\begin{align*}
\pd_t \tw_i &= \pd_x^2 \tw_i,\\
\tw_i(0) &= w_{i,0},\quad i=1,2,
\end{align*}
given by
\begin{equation*}
\tw_i(x,t) = \int_\R K(x-y,t) w_{i,0}(y)dy,\qquad x \in \R, \; t> 0, \; i=1,2,
\end{equation*}
where
\begin{equation*}
K(x,t) = \frac{1}{\sqrt{4 \pi t}} \exp\left (-\frac{x^2}{4t} \right ),
\end{equation*}
is the heat kernel. From standard theory we clearly have the estimates
\begin{align*}
\| \tw_1(t) \|_{s+1} &\leq \| \rho_0 - \rc\|_{s+1},\\
\int_0^T \| \pd_t \tw_1 \|^2_s \, dt + \int_0^T \| \tw_1 \|^2_{s+2} \, dt &\leq C \| \rho_0 - \rc\|^2_{s+1},\\
\| \tw_2(t) \|_s &\leq \| m_0 - \mc\|_s,\\
\int_0^T \| \pd_t \tw_2 \|^2_{s-1} \, dt + \int_0^T \| \tw_2 \|^2_{s+1} \, dt &\leq C \| m_0 - \mc\|^2_s.
\end{align*}
Now, let us introduce the new variable $\hw = w - \tw$, where $\hw = (\hat{\rho}, \hat{m})^\top$, and let us define the operator
\begin{equation*}
\cM(\hw) \hw := \tL(\tw + \hw)\hw + ( \tL(\tw + \hw) - \tL(\tw))\tw.
\end{equation*}
Then, the system \eqref{eq:quasilinear} becomes
\begin{equation}
\left\{
\begin{aligned}
\cM(\hw)\hw &= f(\hw),\\
\hw(0) &= 0,
\end{aligned}
\right.
 \label{eq:sys_homogeneous_IC}
\end{equation}
where $f = -\tL(\tw)\tw$. We will solve \eqref{eq:sys_homogeneous_IC} by iteration. Set
\begin{equation}
\left\{
\begin{aligned}
\cM(\hw_j)\hw_{j+1} &= f,\\
\hw_{j+1}(0) &= 0,
\end{aligned}
\right.
 \label{eq:iteration}
\end{equation}
for $j \in \mathbb{N}_0$ and with $\hw_0 = 0$. Note that the operator $\cM$ in \eqref{eq:iteration} has the same structure as $\tL$ defined in \eqref{operator_time_derivative} and Theorem \ref{theorem_existence_linear_system} applies to the system \eqref{eq:iteration}.
We show how to treat three of the terms in the proof of \eqref{eq:integ_estimate_0_order} for equation \eqref{eq:iteration}. Indeed, we have
\begin{equation*}
\int_\R (\tw_2)_t m \, dx \leq \| (\tw_2)_t \|_{-1} \| m \|_1 \leq \frac{1}{2 \e_1} \| (\tw_2)_t\|^2_{-1} + \frac{\e_1}{2} \| m \|_0^2 + \frac{\e_1}{2} \| m_x \|_0^2.
\end{equation*}
Moreover,
\begin{equation*}
\int_\R (\tw_2)_{xx} m \, dx = \int_\R (\tw_2)_t m \, dx,
\end{equation*}
and, furthermore,
\begin{align*}
\int_\R (\tw_1)_{xxx} m \, dx &= \int_{\R} \pd_t (\tw_1)_x m \, dx \\
&\leq \| \pd_t (\tw_1)_x\|_{-1} \| m \|_{1}\\
&\leq \frac{1}{2 \e_2} \| \pd_t (\tw_1)_x \|^2_{-1} + \frac{\e_2}{2} \| m \|_0^2 + \frac{\e_2}{2} \| m_x \|_0^2.
\end{align*}
The other terms are treated similarly.

Let $i \in \mathbb{N}_0$. Thanks to the Sobolev embedding theorem if $T,\mbox{ } \delta > 0$ are sufficiently small and $\vertiii{\hw}_{s,[0,T]}^2 \leq \delta$, then $\br + \tw_1(x,t) + \hat{\rho}(x,t) > 0$ for $x \in \mathbb{R}$ and $0 \leq t \leq T$. Therefore, it suffices to show that there exist $T > 0$ and sufficiently small $\delta > 0$ such that that the successive iterations satisfy
\begin{align}
\vertiii{\hw_i}_{s,[0,T]}^2 &\leq \delta \leq \beta_s, \qquad i \in \mathbb{N}_0, \label{iteration_ineq_1}\\
\vertiii{\hw_i - \hw_{i-1}}_{s-2,[0,T]} &\leq \frac{1}{2} \vertiii{\hw_{i-1} - \hw_{i-2}}_{s-2,[0,T]},\qquad i \geq 2 \label{iteration_ineq_2},
\end{align}
Suppose that \eqref{iteration_ineq_1} and \eqref{iteration_ineq_2} are satisfied for $i \leq j$. The inequality \eqref{integral_estimate_higher_order} with $n = s$ implies that
\begin{equation*}
\vertiii{\hw_{j+1}}_{s,[0,T]}^2 \leq  C_n(T) \int_{0}^{T} \left ( \| f \|_s^2 + \| f_1 \|_{s+1}^2 \right ) dt.
\end{equation*}
Therefore we can choose $T > 0$ small enough such that $\vertiii{\hw_{j+1}}_{s,[0,T]}^2 \leq \delta \leq \beta_s$.

Now, for $i \in \mathbb{N}_0$ define $v_i := \hw_{i+1} - \hw_i$. Then, $v_i$ satisfies
\begin{equation*}
\left\{
\begin{aligned}
\cM(\hw_i)v_i &= \big (\cM(\hw_{i-1}) - \cM(\hw_i) \big ) \hw_i,\\
v_i(0) &= 0.
\end{aligned}
\right.
\end{equation*}
Moreover, since $H^n(\R)$ is a Banach algebra for $n \geq 1$, we have
\begin{equation*}
\left \| \big (\cM(\hw_{j-1}) - \cM(\hw_j) \big ) \hw_j  \right \| _{s-2}^2 \leq C_{s-2} \delta \| \hat{\rho}_{j} - \hat{\rho}_{j-1} \|_{s-1}^2 + C_{s-2} \delta \| \hat{m}_{j} - \hat{m}_{j-1} \|_{s-2}^2.
\end{equation*}
Applying \eqref{integral_estimate_higher_order} with $n = s- 2$ we obtain
\begin{equation*}
\vertiii{\hw_{j+1} - \hw_j }_{s-2,[0,T]} \leq \tilde{C}_{s-2}(T) \sqrt{\delta} \vertiii{\hw_j - \hw_{j-1} }_{s-2,[0,T]}.
\end{equation*}
Choosing $\delta > 0$ such that $\tilde{C}_{s-2}(T) \sqrt{\delta} \leq 1/2$ concludes the proof of \eqref{iteration_ineq_1} and \eqref{iteration_ineq_2}.
\end{proof}
Thanks to the Sobolev embedding we infer Theorem \ref{themlocale} from Lemma \ref{lemma_local_existence}. This concludes the proof of Theorem \ref{themlocale}.

\section{Linear decay rates}
\label{seclinear}

In this section, we establish the decay of solutions to the linearization of system \eqref{QHD} around an arbitrary constant equilibrium state $U_* = (\rho_*, m_*) \in \R^2$, with $\rho_* > 0$, and satisfying the subsonicity assumption
\begin{equation}
\label{subsonic}
p'(\rho_*) > \frac{m_*^2}{\rho_*^2}.
\end{equation}

In contrast to the estimates from Section \ref{seclocale}, here we focus on \emph{stability} estimates. To that end, we examine the decay structure of the system in the sense of Humpherys' analysis for linear higher order systems in the Fourier space (cf. \cite{Hu05}). Symbol symmetrizability and the existence of an appropriate compensating matrix symbol are key ingredients to establish the optimal decay of the semigroup.

\subsection{Linearization and symbol symmetrizability}

We start by observing that system \eqref{QHD} can be recast in conservation form. Indeed, following Lattanzio \emph{et al.} \cite{LMZ20b} let us write the Bohm potential as
\[
\rho\left(\displaystyle\frac{(\sqrt{\rho})_{xx}}{\sqrt{\rho}}\right)_x = \frac{1}{2} \Big( \rho \big( \ln \rho \big)_{xx}\Big)_x.
\]
Therefore, system \eqref{QHD} in conservation form reads
\begin{equation}
\label{QHDc}
	\begin{cases}
		\rho_t + m_x=0,\\
		m_t +\left(\displaystyle\frac{m^2}{\rho}+p(\rho)\right)_x=\mu m_{xx} + \tfrac{1}{2} k^2 \Big( \rho \big( \ln \rho \big)_{xx}\Big)_x.
	\end{cases}
\end{equation}
The conservative form of the equations will play an important role in the establishment of the energy estimates with the appropiate regularity for the density variable.

Consider an arbitrary constant equilibrium state $U_* = (\rho_*, m_*) \in \R^2$ satisfying $\rho_* > 0$ and \eqref{subsonic}, and let $(\rho + \rho_*, m + m_*)$ be a solution to \eqref{QHDc} where $\rho$ and $m$ represent perturbations. Substituting into \eqref{QHDc} and after some elementary algebra, one arrives at a nonlinear perturbation system of the form
\begin{equation}
\label{QHDp}
	\begin{cases}
		\rho_t + m_x=0,\\
		m_t + \left( p'(\rho_*) - \displaystyle{\frac{m_*^2}{\rho_*^2}}\right) \rho_x + \left( \displaystyle{\frac{2m_*}{\rho_*}}\right) m_x = \mu m_{xx} + \tfrac{1}{2} k^2 \rho_{xxx} + \partial_x N_2,
	\end{cases}
\end{equation}
where $N_2$ contains the nonlinear terms and is of the form
\begin{equation}
\label{orderN2}
N_2 = O\big(\rho^2 + m^2 + \rho_x^2 + |\rho||\rho_{xx}|\big),
\end{equation}
as the reader may easily verify. In other words, the system \eqref{QHD} can be rewritten as a system of the form
\begin{equation}
\label{nonlinQHD}
U_t = \cA U + \partial_x \begin{pmatrix}0 \\ N_2 \end{pmatrix},
\end{equation}
in terms of the (perturbed) state variables $U = (\rho, m)^\top$ and where $\cA$ is a differential operator with constant coefficients. Notice that the nonlinear terms are expressed in conservative form.

Let us consider the linear part of system \eqref{QHDp}, which reads
\begin{equation}
\label{linearQHD}
U_t + A_* U_x = B_* U_{xx} + C_* U_{xxx},
\end{equation}
where 
\begin{equation}
\label{defcoeffs}
\begin{aligned}
U = \begin{pmatrix} \rho \\ m \end{pmatrix}, &\qquad A_* = \begin{pmatrix} 0 & 1 \\ p'(\rho_*) - m_*^2 / \rho_*^2 & 2m_* / \rho_* \end{pmatrix},\\
B_* = \begin{pmatrix} 0 & 0 \\ 0 & \mu \end{pmatrix}, &\qquad C_* = \begin{pmatrix} 0 & 0 \\ \tfrac{1}{2}k^2 & 0 \end{pmatrix},
\end{aligned}
\end{equation}
or, equivalently,
\begin{equation}
\label{defcA}
U_t = \cA U := \big( -A_* \partial_x + B_* \partial_x^2 + C_* \partial_x^3\big) U.
\end{equation}
Take the Fourier transform of \eqref{linearQHD}. This yields 
\begin{equation}
\label{FouQHD}
\hU_t + \big( i \xi A_* + \xi^2 B_* + i \xi^3 C_* \big) \hU = 0,
\end{equation}
where $\hU = \hU(\xi,t)$ denotes the Fourier transform of $U$. The evolution of the solutions to \eqref{FouQHD} reduces to solving the spectral equation 
\begin{equation}
\label{spect}
\big( \lambda I + i \xi A_* + \xi^2 B_* + i \xi^3 C_* \big)\hU = 0,
\end{equation}
for $\lambda \in \C$ and $\xi \in \R$ denoting time frequencies and (Fourier) wave number, respectively. It is said that the linear operator $\cA$ is \emph{strictly dissipative} if for each $\xi \neq 0$ then all solutions to the spectral equation \eqref{spect} satisfy $\Re \lambda (\xi) < 0$ (cf. Humpherys \cite{Hu05}; see also \cite{KaSh88a,ShKa85}). 
\begin{remark}
\label{remtypediss}
Ueda \emph{et al.} \cite{UDK12, UDK18} further classify strictly dissipative systems as follows. The linear system is called strictly dissipative of type $(p,q)$, with $p, q \in \Z$, $p, q \geq 0$, provided that the solutions of the spectral problem \eqref{spect} satisfy
\[
\Re \lambda(\xi) \leq - \, \frac{C |\xi|^{2p}}{(1 + |\xi|^2)^{q}}, \qquad \forall \xi \neq 0, 
\]
for some uniform constant $C > 0$. The system is said to be of standard type when $p = q$ \cite{UDK12}, and of regularity-loss type when $p < q$ \cite{UDK18}. Notice that the heat equation is a system with dissipativity of type $(1,0)$. Hence, the third case when $p > q$ is called dissipativity of regularity-gain type \cite{KSX22}. The type of dissipativity will be reflected in the decay rate of the solutions to the linearized system. Notice that strict dissipativity is equivalent to the stability of the essential spectrum of the linearized operator $\cA$ when computed, for example, with respect to the space $L^2(\R)$ of finite energy perturbations.
\end{remark}

The following result justifies the subsonicity assumption \eqref{subsonic} in the study of strict dissipativity.

\begin{lemma}
\label{lemsupersonic}
Suppose that a constant equilibrium state $(\rho_*,m_*) \in \R^2$ with $\rho_* > 0$ is \emph{supersonic}, that is,
\begin{equation}
\label{supersonic}
p'(\rho_*) < \frac{m_*^2}{\rho_*^2}.
\end{equation}
Then the operator $\cA$ defined in \eqref{defcA} violates the strict dissipativity condition. More precisely, there exist certain values of $\xi \in\R$ for which the solutions to the spectral equation \eqref{spect} satisfy $\Re \lambda(\xi) > 0$.
\end{lemma}
\begin{proof}
See Appendix \ref{secappen}.
\end{proof}

Now, we follow Humpherys \cite{Hu05} and split the symbol into even and odd terms. Let us define the symbols,
\[
\begin{aligned}
A(\xi) &:= A_* + \xi^2 C_* = \begin{pmatrix} 0 & 1 \\ p'(\rho_*) -m_*^2 / \rho_*^2 + \tfrac{1}{2} k^2 \xi^2 & 2m_* / \rho_* \end{pmatrix}, \quad &\text{(odd)},\\
B(\xi) &:= \xi^2 B_* = \xi^2 \begin{pmatrix} 0 & 0 \\ 0 & \mu \end{pmatrix}, \quad &\text{(even)}
\end{aligned}
\]
so that the evolution equation \eqref{FouQHD} is recast as
\begin{equation}
\label{FouQHD2}
\hU_t + \big( i \xi A(\xi) + B(\xi) \big) \hU = 0.
\end{equation}

Notice that in the matrix symbol $A(\xi)$ we have gathered the transport and dispersive terms together (the odd part of the symbol), whereas the only dissipation term due to viscosity (the even part of the symbol) is encoded into the matrix $B(\xi)$. In this fashion, Humpherys mimics the algebraic structure of second order (purely viscous) systems of Kawashima and Shizuta \cite{KaSh88a,ShKa85} at the Fourier level. Humpherys thereby introduces the following fundamental concept of symbol symmetrization, which generalizes the standard notion of symmetrizability of Lax and Friedrichs \cite{FLa67,Frd54} (see also Godunov \cite{Godu61a}).

\begin{definition}[Humpherys \cite{Hu05}]
\label{defsymH}
The operator $\cA$ is \emph{symbol symme\-trizable} if there exists a smooth, symmetric matrix-valued function, $S=S(\xi)>0$, positive-definitive, such that $S(\xi)A(\xi)$ and $S(\xi)B(\xi)$ are symmetric, with  $S(\xi)B(\xi)\geq 0$ (positive semi-definite) for all $\xi \in \R$.
\end{definition}  

\begin{remark}
Let us recall that a generic (quasilinear) system of equations of the form $U_t = \sum_j A_j(U) \partial_x^j U$ is said to be symmetrizable in the classical sense of Friedrichs if, for any constant state $\Us$, there exists a symmetric, positive definite matrix $S=S(\Us) > 0$ such that $S(\Us)A_{j}(\Us)$ are all simultaneously symmetric. Clearly, every symmetrizable system in the sense of Friedrichs is symbol symmetrizable, but the converse is not true.
\end{remark}

\begin{lemma}
\label{Sym2Full}
Assume the subsonicity \eqref{subsonic} of the equilibrium state $U_* = (\rhos,\ms)$ with $\rho_* > 0$. Then the linearized QHD system \eqref{linearQHD} is symbol symmetrizable, but not symmetrizable in the sense of Friedrichs. One symbol symmetrizer is of the form
\begin{equation}
\label{symm}
S(\xi) = \begin{pmatrix}
\alpha(\xi) & 0 \\ 0 & 1 
\end{pmatrix} \in C^{\infty}\big( \R; \R^{2\times 2}  \big),
\end{equation} 
where
\begin{equation}
\label{defalpha}
\alpha(\xi) := p'(\rho_*) - \frac{m_*^2}{\rho_*^2}+ \tfrac{1}{2}k^2 \xi^2 > 0, \qquad \xi \in \R.
\end{equation}
\end{lemma}
\begin{proof}
It is easy to verify that the symbol $S(\xi)$ defined in \eqref{symm} is smooth, symmetric and positive definite because of the condition \eqref{subsonic}. By inspection, one thereby obtains
\[
S(\xi)A(\xi) =  \begin{pmatrix} \alpha(\xi) & 0 \\ 0 & 1 \end{pmatrix}\begin{pmatrix} 0 & 1 \\ \alpha(\xi) & 2m_* / \rho_*\end{pmatrix} = \begin{pmatrix} 0 & \alpha(\xi) \\ \alpha(\xi) & 2m_* / \rho_* \end{pmatrix},
\]
and $S(\xi)B(\xi) = B(\xi)$, which are symmetric matrices with $S(\xi)B(\xi) \geq 0$. This easily shows that the operator $\cA$ is symbol symmetrizable. To prove that the system is not Friedrichs symmetrizable, suppose there exists a positive-definite symmetrizer of the form 
\[
S = \begin{pmatrix}
s_{1} & s_{2} \\ s_{2} & s_{3} 
\end{pmatrix}.
\]
Then the condition on $SA_*$ and $SC_*$ to be simultaneously symmetric matrices implies that $S$ cannot be positive definite, as the reader may easily verify. The lemma is proved.
\end{proof}

\begin{remark}
Up to our knowledge, the QHD system \eqref{QHD} is only the third example of a symbol symmetrizable system which is not Friedrichs symmetrizable, apart from the isothermal Navier-Stokes-Korteweg model (cf. \cite{Hu05,PlV22}) and its non isothermal version (the so called Navier-Stokes-Fourier-Korteweg system \cite{PlV23}). The existence of these simple and physically relevant counterexamples exhibit the importance of Definition \ref{defsymH}.
\end{remark}

Henceforth, multiply equation \eqref{FouQHD2} on the left by the symmetrizer $S = S(\xi)$ defined in \eqref{symm} to rewrite it in symmetric form,
\begin{equation}
\label{sFouQHD}
S(\xi) \hU_t + ( i \xi \tiA(\xi) + \tiB(\xi) ) \hU = 0,
\end{equation}
where
\[
\tiA(\xi) := S(\xi) A(\xi) = \begin{pmatrix} 0 & \alpha(\xi) \\ \alpha(\xi) & 2m_* / \rho_* 
\end{pmatrix}, \qquad
\tiB(\xi) := S(\xi) B(\xi) = \xi^2 \begin{pmatrix} 0 & 0 \\ 0 & \mu \end{pmatrix}.
\]

Once the system in Fourier space is put into symmetric form, we can recall the following fundamental notions (cf. \cite{Hu05,KaSh88a,ShKa85}). Let $S$, $\tiA$, $\tiB \in C^{\infty} \left( \R; \R^{2 \times 2} \right)$ be smooth, real matrix-valued functions of the variable $\xi \in \R$. Assume that $S,$ $\tiA$, $\tiB$ are symmetric for all $\xi \in \R$, $S >0$ is positive definite and $\tiB \geq 0$ is positive semi-definite. The triplet $(S, \tiA, \tiB)$ is said to be genuinely coupled if for all $\xi \neq 0$ every vector $V \in \ker \tiB(\xi)$, with $V \neq 0$, satisfies the condition $\big( \varrho S(\xi)  +    \tiA(\xi) \big) V \neq 0$ for any $\varrho \in \R$. In that case we say that the operator $\cA$ satisfies the \emph{genuine coupling condition}. 

Likewise, under the same assumptions of symmetry, smoothness and positive semidefiniteness, if a smooth, real matrix valued function,  $K\in C^{\infty} \left( \R; \R^{3 \times 3} \right)$, satisfies
\begin{itemize}
\item[(a)] $K(\xi)S(\xi)$ is skew-symmetric for all $\xi\in \R$; and,
\item[(b)] $\big[K(\xi)\tiA(\xi)\big]^{s}+ \tiB(\xi) \geq \theta(\xi) I > 0$ for all $\xi \in \R$, $\xi \neq 0$, and some $\theta = \theta(\xi) > 0$,
\end{itemize}
then $K$ is said to be a \emph{compensating matrix symbol} for the triplet $(S, \tiA, \tiB)$. Here $[M]^{s} := \frac{1}{2}(M+M^\top)$ denotes the symmetric part of any real matrix $M$. 

The concepts of strict dissipativity, genuine coupling and the existence of a compensating matrix function are equivalent to each other (see Theorems 3.3 and 6.3 by Humpherys \cite{Hu05}), as it is stated in the following equivalence theorem, under the extra constant multiplicity assumption.
\begin{theorem}[equivalence theorem \cite{Hu05}]
\label{HuThSym}
Suppose that a symbol symmetrizer, $S= S(\xi)$, $S \in C^\infty(\R; \R^{3 \times 3})$, exists for the operator $\mathcal{A}$ in the sense of Definition \ref{defsymH}, and that $\tiA(\xi) = S(\xi)A(\xi)$ is of constant multiplicity in $\xi$, that is, all its eigenvalues are semi-simple and with constant multiplicity for all $\xi \in \R$. Then the following conditions are equivalent:
\begin{itemize}
\item[(a)] $\mathcal{A}$ is strictly dissipative.
\item[(b)] $\mathcal{A}$ is genuinely coupled.
\item[(c)] There exists a compensating matrix function for the triplet $(S,SA, SB)$.
\end{itemize}
\end{theorem}

Our first observation is that the linearized QHD system \eqref{FouQHD2} satisfies the genuine coupling and constant multiplicity conditions.
\begin{lemma}
\label{lemgencoup}
The triplet $(S, \tiA, \tiB)$ is genuinely coupled. Moreover, the matrix symbol $\tiA(\xi)$ is of constant multiplicity in $\xi \in \R$.
\end{lemma}
\begin{proof}
Clearly, for each $\xi \neq 0$ we have $\ker \tiB(\xi) = \{  (a, 0)^\top \, :  \, a \in \R \} \subset \R^2$. Hence, for any $0 \neq V = (a,0)^\top  \in \ker \tiB(\xi)$, with $\xi \neq 0$, and any $\varrho \in \R$, there holds
\[
\big( \varrho S(\xi) + \tiA(\xi) \big) V = \begin{pmatrix} \varrho \alpha(\xi) & \alpha(\xi) \\ \alpha(\xi) & \varrho + 2m_* / \rho_* \end{pmatrix} \begin{pmatrix} a \\ 0\end{pmatrix} = \begin{pmatrix} a \varrho \alpha(\xi) \\ a \alpha(\xi) \end{pmatrix} \neq 0,
\]
because $a \neq 0$ and $\alpha(\xi) > 0$ for all $\xi$. Therefore, the triplet $(S, \tiA, \tiB)$ is genuinely coupled. Upon an explicit computation of the eigenvalues of $\tiA(\xi)$, we obtain that $\det (\nu I - \tiA(\xi)) = 0$ if and only if
\[
\nu = \nu_\pm(\xi) = \frac{m_*}{\rho_*} \pm \sqrt{\frac{m_*^2}{\rho_*^2} + \alpha(\xi)^2},
\]
yielding two real and simple eigenvalues, $\nu_-(\xi) < 0 < \nu_+(\xi)$, which never coalesce inasmuch as $\alpha(\xi) > 0$ for all $\xi$. Hence, the constant multiplicity assumption is also fulfilled.
\end{proof}

\subsection{The compensating matrix symbol}

By virtue of the equivalence Theorem \ref{HuThSym} and Lemma \ref{lemgencoup}, we deduce the existence of a compensating matrix symbol for the triplet $(S, \tiA, \tiB)$ associated to the linear symmetric QHD system \eqref{sFouQHD}. In applications, however, it is more convenient to construct the symbol directly. For that purpose, let us go back to the (unsymmetrized) original system \eqref{FouQHD2} and introduce the following rescaling of variables,
\begin{equation}
\label{hV}
\hV := S(\xi)^{1/2} \hU,
\end{equation}
where $S(\xi) > 0$ is the symmetrizer from Lemma \ref{Sym2Full}. Hence, system \eqref{FouQHD2} transforms into
\begin{equation}
\label{eqforV}
\hV_{t} + \big( i \xi \hA(\xi) + \hB(\xi) \big) \hV = 0, 
\end{equation}
where
\[
\hA(\xi) := S(\xi)^{1/2} A(\xi) S(\xi)^{-1/2},\qquad \hB(\xi) := S(\xi)^{1/2} B(\xi) S(\xi)^{-1/2}.
\]
Thus, direct computations yield
\[
\hA(\xi) = \begin{pmatrix} 0 & \alpha(\xi)^{1/2} \\ \alpha(\xi)^{1/2} & 2m_* / \rho_* \end{pmatrix}, \qquad \hB(\xi) = \xi^2 \begin{pmatrix} 0 & 0 \\ 0 & \mu \end{pmatrix} = \xi^2 B_*.
\]
Notice that $\hA$ and $\hB$ are both smooth and symmetric, with $\hB \geq 0$. In some sense, we have just symmetrized the original system with $S(\xi)^{1/2}$ instead. The following lemma appropriately chooses the compensating matrix function and provides more information than the equivalence theorem.

\begin{lemma}
\label{lemourK}
There exists a smooth compensating matrix symbol, $\hK \in C^\infty(\R; \R^{2 \times 2})$, $\hK = \hK(\xi)$, for the triplet $(I, \hA(\xi), B_*)$. In other words, $\hK$ is skew-symmetric and 
\begin{equation}
\label{compmatprop}
[\hK(\xi) \hA(\xi)]^s + B_* \geq {\theta} I > 0, 
\end{equation}
for some uniform constant ${\theta} > 0$ independent of $\xi \in \R$. In addition, the following estimates hold,
\begin{equation}
\label{tiKbded}
|\xi \hK(\xi)|, | \hK(\xi) | \leq C,
\end{equation}
for all $\xi \in \R$ and some uniform constant $C>0$. 
\end{lemma} 
\begin{proof}
Let us proceed by inspection. Consider a compensating matrix symbol of the form
\[
\hK(\xi) = \epsilon q(\xi) \begin{pmatrix} 0 & 1 \\ -1 & 0 \end{pmatrix},
\]
with $\epsilon > 0$ constant and $q(\xi) > 0$ real and smooth, both to be chosen later. Clearly, $\hK$ is skew-symmetric. Let us compute
\[
\hK(\xi) \hA(\xi) = \epsilon q(\xi) \begin{pmatrix} 0 & 1 \\ -1 & 0 \end{pmatrix} \begin{pmatrix} 0 & \alpha(\xi)^{1/2} \\ \alpha(\xi)^{1/2} & 2m_* / \rho_* \end{pmatrix} = \epsilon q(\xi) \begin{pmatrix} \alpha(\xi)^{1/2} & 2m_* / \rho_* \\ 0 & -\alpha(\xi)^{1/2} \end{pmatrix}.
\]
The symmetric part of this matrix is
\[
[\hK(\xi) \hA(\xi)]^s = \epsilon q(\xi) \begin{pmatrix} \alpha(\xi)^{1/2} & m_* / \rho_* \\m_* / \rho_* & -\alpha(\xi)^{1/2} \end{pmatrix}.
\]
Hence, let us choose $q(\xi) \equiv \alpha(\xi)^{-1/2}$, real smooth and positive, to obtain for every $y = (y_1, y_2) \in \R^2$ and all $\xi \in \R$ the quadratic form
\[
\begin{aligned}
Q(y,\xi) &= \begin{pmatrix} y_1 \\ y_2   \end{pmatrix}^{\top} \! \big( [\hK(\xi) \hA(\xi)]^s + B_* \big) \begin{pmatrix} y_1 \\ y_2   \end{pmatrix} \\
&=\begin{pmatrix} y_1 \\ y_2   \end{pmatrix}^{\top} \begin{pmatrix} \epsilon & \epsilon \alpha(\xi)^{-1/2} m_* / \rho_* \\ \epsilon \alpha(\xi)^{-1/2} m_* / \rho_* & \mu - \epsilon \end{pmatrix} \begin{pmatrix} y_1 \\ y_2   \end{pmatrix}\\
&= a_1 y_1^2 + b_{12} y_1 y_2 + a_2 y_2^2,
\end{aligned}
\]
with
\[
a_1 = \epsilon > 0, \quad a_2 = \mu -\epsilon, \quad b_{12} = \frac{2 \epsilon m_*}{\alpha(\xi)^{1/2} \rho_*}.  
\]
If we choose $0 < \epsilon \ll 1$ sufficiently small such that
\[
a_2 > 0, \qquad a_2 - \frac{b_{12}^2}{2a_1} > 0,
\]
then clearly,
\[
\begin{aligned}
Q(y,\xi) &= \tfrac{1}{2} a_1 y_1^2 + \tfrac{1}{2} a_1 \Big( y_1 + \frac{b_{12}}{a_2} y_2 \Big)^2 + \Big( a_2 - \frac{b_{12}^2}{2a_1} \Big) y_2^2\\
&\geq \tfrac{1}{2} a_1 y_1^2 + \Big( a_2 - \frac{b_{12}^2}{2a_1} \Big) y_2^2\\
&\geq \theta |y|^2,
\end{aligned}
\]
with $\theta = \min \{ \tfrac{1}{2} a_1, a_2 - b_{12}^2/(2a_1) \} > 0$. That is, the quadratic form is positive. Hence, we need to choose $0 < \epsilon \ll 1$ such that $\mu > \epsilon$ and
\[
a_2 - \frac{b_{12}^2}{2a_1} = \mu - \Big( 1 + \frac{2 m_*^2}{\rho_*^2 \alpha(\xi)}\Big) \epsilon > 0.
\]
Notice, however, that $\alpha(\xi) \geq \alpha_* > 0$ for all $\xi \in \R$ with
\[
\alpha_* := p'(\rho_*) - \frac{m_*^2}{\rho_*^2} > 0,
\]
which is a positive constant because of the subsonicity condition \eqref{subsonic}. Therefore,
\[
a_2 - \frac{b_{12}^2}{2a_1} \geq \mu - \Big( 1 + \frac{2 m_*^2}{\rho_*^2 \alpha_*}\Big) \epsilon > 0,
\]
for all $\xi$ and it suffices to choose
\[
\epsilon = \epsilon_* := \tfrac{1}{2} \frac{\mu \alpha_* \rho_*^2}{\alpha_* \rho_*^2 + 2 m_*^2} > 0,
\]
in order to obtain $Q(y,\xi) \geq \theta |y|^2$ for all $\xi \in \R$ and all $y \in \R^2$ with a constant
\[
\theta = \min \Big\{ \tfrac{1}{2} \epsilon_*, \mu - \epsilon_* \Big( 1 + \frac{2m_*^2}{\alpha_* \rho_*}\Big) \Big\} > 0,
\]
independent of $\xi$. This shows \eqref{compmatprop}. Therefore,
\[
\hK(\xi) = \frac{\epsilon_*}{\alpha(\xi)^{1/2}}\begin{pmatrix} 0 & 1 \\ -1 & 0 \end{pmatrix},
\]
is the compensating symbol we look for. Clearly, $\hK$ is smooth in $\xi$. Finally, since $0 < \alpha(\xi)^{-1/2} \leq \alpha_*^{-1/2}$ and since $|\xi| \alpha(\xi)^{-1/2}$ is uniformly bounded for all $\xi \in \R$, we conclude that there exists a uniform constant $C > 0$ such that \eqref{tiKbded} holds. The lemma is proved.
\end{proof}

\begin{remark}
\label{remgoodK}
A few comments are in order. Notice that we demand $\hK$ to be a compensating matrix symbol for the triplet $(I, \hA, B_*)$ and not for $(I, \hA, \hB)$. This feature will be useful in the establishment of the energy estimate. In addition, we have constructed $\hK$ such that the constant $\theta > 0$ in \eqref{compmatprop} can be chosen uniformly in $\xi \in \R$ and that both $|\hK(\xi)|$ and $|\xi \hK(\xi)|$ are uniformly bounded above. These are properties that cannot be deduced from the equivalence theorem.
\end{remark}

\subsection{Linear decay of the associated semigroup}

\begin{lemma}[basic pointwise estimate]
\label{lembee}
The solutions $\hV = \hV(\xi,t)$ to the linear system \eqref{eqforV} satisfy the pointwise estimate
\begin{equation}
\label{bestV}
|\hV(\xi,t)| \leq C \exp (- \omega_0 \xi^2 t) |\hV(\xi,0)|,
\end{equation}
for all $\xi \in \R$, $t \geq 0$ and some uniform constants $C,\omega_0 > 0$.
\end{lemma}
\begin{proof}
The proof follows that of Lemma 5.2 in \cite{PlV22} almost word by word and therefore we gloss over some of the details. The important points in the present case are the following. By taking the standard product in $\C^2$ and for any $\delta > 0$ sufficiently small, the skew-symmetry of the matrix $\hK$ allows us to define an energy of the form
\[
\cE = |\hV|^2 - \delta \xi \langle \hV, i \hK(\xi) \hV \rangle,
\]
which is real, positive and equivalent to $| \hV |^2$, that is, $C_1^{-1} |\hV|^2 \leq \cE \leq C_1 |\hV|^2$ for some uniform $C_1 > 0$. Since $\hB(\xi) = \xi^2 B_*$, the inner product in $\C^2$ of $\hV$ with equation \eqref{eqforV} yields
\begin{equation}
\label{la8}
\tfrac{1}{2} \partial_t |\hV|^2 + \xi^2 \langle \hV, B_* \hV \rangle = 0,
\end{equation}
where we have used the fact that $\hA$ and $B_*$ are symmetric. Likewise, multiplying the equation by by $- i \xi \hK(\xi)$ one arrives at the estimate
\begin{equation}
\label{la10}
- \tfrac{1}{2} \xi \partial_t \langle \hV, i \hK(\xi) \hV \rangle + \xi^2 \langle \hV, [\hK(\xi) \hA(\xi)]^s \hV \rangle \leq \e \xi^2 |\hV|^2 + C_\e \xi^2 \langle \hV, B_* \hV \rangle,
\end{equation}
for any $\e > 0$ and some uniform $C_\e > 0$. Multiply inequality \eqref{la10} by $\delta > 0$ and add it to equation \eqref{la8}, yielding \begin{equation}
\label{la11}
\begin{aligned}
\tfrac{1}{2} \partial_t \big( |\hV|^2 - \delta \xi \langle \hV, i \hK(\xi) \hV \rangle \big) + \xi^2 \big( \delta \langle \hV, [\hK(\xi) \hA(\xi)]^s \hV \rangle &+ (1-\delta C_\e)  \langle \hV, B_* \hV \rangle \big) \\&\leq \e \delta \xi^2 |\hV|^2.
\end{aligned}
\end{equation}
If we choose $\e = \tfrac{1}{2} \theta$ where $\theta > 0$ is the uniform constant in \eqref{compmatprop} then the constant $C_\e > 0$ is therefore fixed and for $0 < \delta \ll 1$ sufficiently small one obtains
\[
\delta \langle \hV, [\hK(\xi) \hA(\xi)]^s \hV \rangle + (1-\delta C_\e)\langle\hV, \tiB \hV \rangle \geq \delta \langle \hV, ([\hK(\xi) \hA(\xi)]^s + B_*) \hV \rangle 
\geq \delta \theta |\hV|^2,
\] 
where we have used the main property of the compensating matrix symbol (estimate \eqref{compmatprop}). Substitution into \eqref{la11} yields
\[
\partial_t \cE + \omega_0 \xi^2 \cE \leq 0,
\]
where $\omega_0 := \delta \theta / C_1 > 0$. This implies estimate \eqref{bestV}. Details are left to the reader.
\end{proof}
\begin{remark}
It is to be noticed that \eqref{bestV} implies that the eigenvalues in Fourier space of system \eqref{spect} satisfy 
$\lambda(\xi) \leq  - \omega_0 \xi^2$, with $\omega_0 > 0$, yielding a dissipative structure of regularity-gain type.
\end{remark}

The pointwise estimate of Lemma \ref{lembee} implies the following estimate for the solutions to \eqref{FouQHD2}.

\begin{corollary}
\label{corlindecayW}
The solutions $\hU(\xi,t) = (\hU_1, \hU_2)(\xi,t)$ to the linear system \eqref{FouQHD2} satisfy the estimate 
\begin{equation}
\label{beeU}
\begin{aligned}
(1+\xi^{2})| \hU_{1}(\xi, t) |^{2} + &| \hU_{2}(\xi, t) |^{2}  \leq  C \exp(- 2\omega_0 \xi^{2} t ) \big( (1+\xi^{2})| \hU_{1}(\xi, 0) |^{2} + | \hU_{2}(\xi, 0) |^{2}\big),  
\end{aligned}
\end{equation}
for all $t\geq 0$, $\xi \in \R $ and some uniform constant $C>0$.
\end{corollary}
\begin{proof}
Suppose $\hU = \hU(\xi,t)$ is a solution to system \eqref{FouQHD2} . Then from transformation \eqref{hV} we know that $\hV = S(\xi)^{1/2} \hU$ satisfies \eqref{eqforV} and, therefore, Lemma \ref{lembee} applies. Hence, from estimate \eqref{bestV} we obtain
\[
\begin{aligned}
|\hV|^2 = | S(\xi)^{1/2} \hU |^2 &= \left| \begin{pmatrix}
\alpha(\xi)^{1/2} & 0 \\ 0 & 1 
\end{pmatrix} \begin{pmatrix} \hU_1 \\ \hU_2\end{pmatrix} \right|^2 = \alpha(\xi) |\hU_1|^2 +  |\hU_2|^2 \\
&\leq C \exp (- 2\omega_0 \xi^2 t) |\hV(\xi,0)|^2\\
&= C \exp (- 2\omega_0 \xi^2 t) \big( \alpha(\xi) |\hU_1(\xi,0)|^2 +  |\hU_2(\xi,0)|^2\big).
\end{aligned}
\]
From the definition of $\alpha(\xi)$ (see \eqref{defalpha}) we clearly deduce that there exist constants $C_j > 0$ such that $C_2 (1+\xi^2) \leq \alpha(\xi) \leq C_1(1+\xi^2)$ for all $\xi \in \R$. Upon substitution we obtain estimate \eqref{beeU}.
\end{proof}

The decay estimates \eqref{beeU} of the solutions to the evolution equation in Fourier space  \eqref{FouQHD2} readily imply the decay of the semigroup associated to the linear evolution system \eqref{linearQHD}. Notice the higher regularity on the density variable (here $U_1 = \rho$) that appears in \eqref{beeU}. Consider the abstract Cauchy problem for the linear system \eqref{linearQHD},
\begin{equation}
\label{Cauchylin}
\left\{
\begin{aligned}
		U_t &= \cA U,\\
		U(0) &= f,
\end{aligned}
\right.
\end{equation}
where $\cA := -A_* \partial_x + B_* \partial_x^2 + C_* \partial_x^3$ is a differential operator with constant coefficients. We densely define the operator on the space $Z := L^2(\R) \times L^2(\R)$ with domain $D(\cA) = H^{s+1}(\R) \times H^s(\R)$ for some $s \geq 3$.

\begin{lemma}
\label{lemsg}
The differential operator $\cA : Z \to Z$ is the infinitesimal generator of a $C_0$-semigroup, $\{ e^{t\cA} \}_{t\geq 0}$, in $Z = L^2(\R) \times L^2(\R)$. Moreover, for any $f \in \big( H^{s+1}(\R) \times H^s(\R) \big) \cap \big( L^1(\R) \times L^1(\R) \big)$, $s \geq 2$, and all $0 \leq \ell \leq s$, $t > 0$, there holds the estimate
\begin{equation}
\label{linestsg}
\begin{aligned}
\Big( \| \partial_x^{\ell} (e^{t \cA} f)_1(t) \|_1^2 + \| \partial_x^{\ell} (e^{t \cA} f)_2(t) \|_0^2 \Big)^{1/2} &\leq C e^{-c_1t} \Big( \| \partial_x^{\ell } f_1 \|_1^2 + \| \partial_x^{\ell} f_2 \|_0^2 \Big)^{1/2}  +\\ & \;\; + C (1+t)^{-(\ell/2 + 1/4)} \| f \|_{L^1},
\end{aligned}
\end{equation}
for some uniform constants $C, c_1 > 0$.
\end{lemma}
\begin{proof}
The infinitesimal semigroup generated by $\cA$ is necessarily associated to the solutions to the linear problem \eqref{Cauchylin}, which can be expressed in terms of the inverse Fourier transform of the solutions to \eqref{FouQHD2}. Indeed, suppose that $\hU = \hU(\xi,t)$ is the solution to \eqref{FouQHD2} with initial condition $\hU(\xi,0) = \widehat{f}(\xi)$. Then $U(x,t) = (e^{t \cA} f)(x)$ is the solution to \eqref{Cauchylin} with $U(0) = f = (f_1, f_2)^\top$, where
\begin{equation}
\label{repsg}
(e^{t \cA} f)(x) := \frac{1}{\sqrt{2 \pi}} \int_\R e^{i x \xi}e^{t R(i\xi)} \hat{f}(\xi) \, d \xi,
\end{equation}
and
\[
R(z) := - \begin{pmatrix} 0 & z \\ z \big( p'(\rho_*) -m_*^2 / \rho_*^2 - \tfrac{1}{2} k^2 z^2\big) & 2zm_* / \rho_* + z^2 \mu \end{pmatrix}, \quad z \in \C,
\]
\[
R(i\xi) = - (i \xi A(\xi) + B(\xi)), \qquad \xi \in \R.
\]
That $\{ e^{t\cA}\}_{t \geq 0}$ is a $C_0$-semigroup, where $\cA$ is the constant coefficient differential operator defined above, follows from standard Fourier estimates and semigroup theory (cf. \cite{Pazy83,EN06}); we omit the details. Now, since $\hU$ satisfies \eqref{FouQHD2}, then by Corollary \ref{corlindecayW} estimate \eqref{beeU} holds. Fix $\ell \in [0,s]$, multiply \eqref{beeU} by $\xi^{2 \ell}$ and integrate in $\xi \in \R$. This yields
\[
\int_\R \Big[ \xi^{2 \ell}(1+\xi^2) |\hU_1(\xi,t)|^2 + \xi^{2 \ell} |\hU_2(\xi,t)|^2 \Big] \, d\xi \leq C J_1(t) + C J_2(t),
\]
where,
\[
\begin{aligned}
J_1(t) &:= \int_{-1}^1 \Big[ \xi^{2 \ell}(1+\xi^2) |\hU_1(\xi,0)|^2 + \xi^{2 \ell} |\hU_2(\xi,0)|^2 \Big] \exp ( - 2\omega_0 \xi^2 t) \, d \xi,\\
J_2(t) &:= \int_{|\xi|\geq 1} \Big[ \xi^{2 \ell}(1+\xi^2) |\hU_1(\xi,0)|^2 + \xi^{2 \ell} |\hU_2(\xi,0)|^2 \Big]\exp ( - 2\omega_0 \xi^2 t) \, d \xi.
\end{aligned}
\]
Noticing that, clearly, $\exp (-2 \omega_0 \xi^2 t) \leq \exp (-\omega_0 \xi^2 t)$, we deduce
\[
J_1(t) \leq 2 \int_{-1}^1 \xi^{2 \ell} |\hU(\xi,0)|^2 e^{-\omega_0t \xi^2} \, d\xi \leq 2 \sup_{\xi \in \R} |\hU(\xi,0)|^2 \int_{-1}^1 \xi^{2 \ell} e^{-\omega_0t \xi^2} \, d\xi.
\]
But since for any fixed $\ell \in [0,s]$ and any constant $\omega_0 > 0$, the integral
\[
H_0(t) := (1+t)^{\ell + 1/2} \int_{-1}^1 \xi^{2 \ell} e^{-\omega_0t \xi^2} \, d\xi \leq C,
\]
is uniformly bounded for all $t > 0$ with some constant $C > 0$ (see Lemma A.1 in \cite{PlV22}), we arrive at
\begin{equation}
\label{estJ1}
J_1(t) \leq C (1+t)^{-(\ell + 1/2)} \| U(x,0)\|_{L^1}^2.
\end{equation}

Now, if $|\xi| \geq 1$ then $\exp (-2 \omega_0 t\xi^2) \leq \exp(-\omega_0 t)$. Hence, Plancherel's theorem implies that
\[
\begin{aligned}
J_2(t) &\leq e^{-\omega_0t} \int_{|\xi|\geq 1} \xi^{2 \ell}(1+\xi^2) |\hU_1(\xi,0)|^2 + \xi^{2 \ell} |\hU_2(\xi,0)|^2 \, d\xi \\
&= e^{-\omega_0t} \int_{|\xi|\geq 1} ( \xi^{2\ell}+ \xi^{2(\ell +1)} ) |\hU_1(\xi,0)|^2 + \xi^{2 \ell} |\hU_2(\xi,0)|^2 \, d\xi \\
&\leq  e^{-\omega_0t} \int_\R (\xi^{2\ell}+ \xi^{2(\ell+1)}) |\hU_1(\xi,0)|^2 + \xi^{2 \ell} |\hU_2(\xi,0)|^2 \, d\xi \\
&=  e^{-\omega_0t} \big( \| \partial_x^{\ell} U_1(0) \|_1^2 + \| \partial_x^{\ell} U_2(0) \|_0^2 \big),
\end{aligned}
\]
for all $t > 0$. Combining both estimates we obtain the result with $c_1 = \omega_0/2 > 0$, as claimed.
\end{proof}

\section{Global existence and decay of perturbations of equilibrium states}
\label{secglobal}

In this section we focus on the nonlinear problem. We prove the global existence and the decay of small perturbations to subsonic equilibrium states.

\subsection{Nonlinear energy estimates}

We start by establishing a priori energy estimates for solutions to the full nonlinear problem \eqref{QHD}. Let $U_* = (\rho_*, m_*) \in \R^2$ be a subsonic equilibrium state with $\rho_* > 0$. Then if $(\rho + \rhos,m + \ms)$ solves \eqref{QHD} with $U = (\rho,m)$ being a perturbation, then the latter satisfies the equivalent nonlinear system \eqref{nonlinQHD} where $\cA$ is the linearized operator around $U_*$ defined in \eqref{defcA}. Let us denote the initial perturbation as $(\rho_0, m_0)$ and suppose that
\[
\rho_0 \in H^{s+1}(\R) \cap L^1(\R), \qquad m_0 \in H^{s}(\R) \cap L^1(\R),
\]
for some $s \geq 3$. From the local existence theorem \ref{themlocale} we know that if $E_s(0)^{1/2} = (\| \rho_0 \|_{s+1}^2 + \| m_0 \|_s^2 )^{1/2} < a_0$ then there exists a local solution to system \eqref{nonlinQHD} in the perturbation variables, namely, $U = (\rho,m) \in X_s((0,T); r, R)$, for some $R \geq r > 0$, $T > 0$, and with initial condition $U(0) = U_0 := (\rho_0, m_0)$, such that $U + U_* = (\rho + \rhos, m + \ms)$ solves the original system \eqref{QHD} with initial condition $U(0) + U_* = (\rho_0 + \rho_*, m_0 + \ms)$. This local solution can be written in terms of the associated semigroup and the variations of constants formula,
\begin{equation}
\label{SemSol}
U(x,t)= e^{t \cA}U_{0} + \int_{0}^{t} e^{(t- \tau)\cA} \partial_x \begin{pmatrix} 0 \\ N_2\end{pmatrix}(\tau) \, d\tau.
\end{equation}

Therefore, for any fixed $0 \leq \ell \leq s-1$ we apply the decay estimates for the semigroup (see Lemma \ref{lemsg} and estimate \eqref{linestsg}) to obtain
\begin{equation}
\label{star1}
\begin{aligned}
\Big( \Vert \partial_{x}^{\ell}\rho(t) \Vert_{1}^{2} + \Vert \partial_{x}^{\ell}m(t) \Vert_{0}^{2} \Big)^{1/2} 
&\leq C  e^{-c_{1}t}\Big( \Vert  \partial_{x}^{\ell}\rho_{0} \Vert_{1}^{2} +\Vert \partial_{x}^{\ell}m_{0} \Vert^{2}_0 \Big)^{1/2} + \\
&\quad + C( 1+t )^{-(1/4 + \ell/2)} \big( \| \rho_0 \|_{L^1} + \|m_0 \|_{L^1}\big)  +\\
&\quad + C\int_{0}^{t} \Vert \partial_{x}^{\ell}(e^{(t-\tau)\cA} \partial_x N_{2}(\tau) )\Vert_0 \, d\tau . 
\end{aligned}
\end{equation}

From the representation of the semigroup in \eqref{repsg}, which leads to the expression $\hU = e^{-t R(i\xi)} \hU(0)$ for any solution to the linear problem, it is easy to verify the following identity,
\[
\partial_x^{\ell} \big( e^{t \cA} \partial_x f \big) = \partial_x^{\ell+1} \big( e^{t \cA} f\big),
\]
for any $f\in H^{s}(\R)$, $0 \leq \ell \leq s-1$ and $t \geq 0$; details are left to the reader. Therefore, we may apply estimate \eqref{linestsg} once again, but now with $\ell + 1 \leq s$ replacing $\ell$, in order to arrive at
\begin{equation}
\label{star2}
\begin{aligned}
\int_{0}^{t} \Vert \partial_{x}^{\ell}(e^{(t-\tau)\cA} \partial_x N_{2}(\tau) \Vert_0 \: d\tau &\leq C \int_0^t e^{-c_1(t-\tau)} \| \partial_x^{\ell+1} N_2(\tau) \|_0 \, d\tau + \\
&\quad + C \int_0^t (1+t-\tau)^{\ell/2 + 3/4} \| N_2(\tau) \|_{L^1} \, d\tau.
\end{aligned}
\end{equation}
Notice that the particular (conservative) form of the nonlinear term, namely $\partial_x (0, N_2)^\top$, is crucial to obtain the algebraic time decay inside that last integral. Upon substitution we obtain
\begin{equation}
\label{DerEstNlop}
\begin{aligned}
\Big( \Vert \partial_{x}^{\ell}\rho(t) \Vert_{1}^{2} + \Vert \partial_{x}^{\ell}m(t) \Vert_{0}^{2} \Big)^{1/2} 
&\leq C  e^{-c_{1}t}\Big( \Vert  \partial_{x}^{\ell}\rho_{0} \Vert_{1}^{2} +\Vert \partial_{x}^{\ell}m_{0} \Vert^{2}_0 \Big)^{1/2} + \\
&\quad + C( 1+t )^{-(1/4 + \ell/2)} \big( \| \rho_0 \|_{L^1} + \|m_0 \|_{L^1}\big)  +\\
&\quad + C \int_0^t e^{-c_1(t-\tau)} \| \partial_x^{\ell+1} N_2(\tau) \|_0 \, d\tau + \\
&\quad + C \int_0^t (1+t-\tau)^{\ell/2 + 3/4} \| N_2(\tau) \|_{L^1} \, d\tau,
\end{aligned}
\end{equation}
for all $0 \leq \ell \leq s-1$. Summing up estimates \eqref{DerEstNlop} for $\ell = 0, 1, \ldots, s-1$ yields
\[
\begin{aligned}
\| \rho(t) \|_s + \| m(t) \|_{s-1} &\leq C e^{-c_1 t} \big( \| \rho_0 \|_s + \|m_0 \|_{s-1} \big) + C (1+t)^{-1/4} \big( \| \rho_0 \|_{L^1} + \|m_0 \|_{L^1}\big) +  \\
&\quad + C \int_0^t e^{-c_1(t - \tau)} \| N_2(\tau) \|_s \, d \tau \, + \\
&\quad + C \int_0^t (1+t - \tau)^{-3/4} \| N_2(\tau) \|_{L^1} \, d \tau.
\end{aligned}
\]
Since, clearly, there exists a uniform constant $C > 0$ such that $e^{-c_1 t} \leq C(1+t)^{-1/4}$ for all $t \geq 0$, we simplify last estimate as
%\begin{equation}
\begin{align}
\| \rho(t) \|_s + \| m(t) \|_{s-1} &\leq C (1+t)^{-1/4} \Big( \| \rho_0 \|_s + \|m_0 \|_{s-1}  + \| \rho_0 \|_{L^1} + \|m_0 \|_{L^1}\Big) +  \nonumber\\
&\quad + C \int_0^t e^{-c_1(t - \tau)} \| N_2(\tau) \|_s \, d \tau \, + \label{estimatef}\\
&\quad + C \int_0^t (1+t - \tau)^{-3/4} \| N_2(\tau) \|_{L^1} \, d \tau. \nonumber
\end{align}
%\end{equation}

Now, we proceed with the estimation of the nonlinear terms. For that purpose we first recall some classical results.

\begin{lemma}
\label{lemaux1}
$\,$
\begin{itemize}
%\item[(a)] For any $s > n/2$, $n \in \N$, the space $H^s(\R^n)$ is a Banach algebra. Moreover, there exists a constant $C_s > 0$ such that
%\[
%\| uv \|_s \leq C_s \|u\|_s \|v \|_s,
%\] 
%for all $u,v \in H^s(\R^n)$.
\item[(a)] Let $u, v \in H^s(\R^n) \cap L^\infty(\R^n)$, for any $s \in \R$, $n \in \N$. Then
\[
\| uv \|_s \leq C \big( \| u \|_s \| v \|_\infty + \|u \|_\infty \|v \|_s \big),
\]
for some uniform constant $C > 0$.
\item[(b)] Let $s \geq 0$ and $k \geq 0$ be such that $s+k \geq [\frac{n}{2}] +1$. Assume that $u \in H^s(\R^n)$, $v \in H^k(\R^n)$. Then for $\ell = \min \{ s, k, s+k - [\frac{n}{2}] -1 \}$ we have $uv \in H^\ell(\R^n)$ and there exists a uniform $C_{s,k} > 0$ such that
\[
\| uv \|_\ell \leq C_{s,k} \| u \|_s \| v \|_k.
\]
In particular, if $s \geq [\frac{n}{2}] +1$, $0 \leq \ell \leq s$ and $u \in H^s(\R^n)$, $v \in H^\ell(\R^n)$, then
\[
\| uv \|_\ell \leq C_s \| u \|_s \| v \|_\ell,
\]
for some uniform $C_s > 0$.
\end{itemize}
\end{lemma}
\begin{proof}
For the proof of (a) see Lemma 3.2 in \cite{HaLi96a}. The proof of (b) is a corollary of the interpolation inequalities obtained by Nirenberg \cite{Nir59} (see also Corollary 2.2 and Lemmata 2.1 and 2.3 in \cite{KaTh83}). 
\end{proof}

\begin{corollary}
\label{corBanachalg}
For any $s > n/2$, $n \in \N$, the space $H^s(\R^n)$ is a Banach algebra. Moreover, there exists a constant $C_s > 0$ such that
\[
\| uv \|_s \leq C_s \|u\|_s \|v \|_s,
\] 
for all $u,v \in H^s(\R^n)$.
\end{corollary}
\begin{proof}
Follows immediately from Lemma \ref{lemaux1} (b).
\end{proof}

We also need some estimates on composite functions.

\begin{lemma}
\label{lemaux2}
Let $s \geq 1$ and suppose that $Y = (Y_1, \ldots, Y_m)$, $m \in \N$, $Y_i \in H^s(\R^n) \cap L^\infty(\R^n)$, for all $1 \leq i \leq m$. Let $\Lambda = \Lambda(Y)$, $\Lambda : \R^m \to \R^m$, be a $C^\infty$ function. Then for each $1 \leq j \leq s$ there hold $\partial_x \Lambda(Y) \in H^{j-1}(\R^n)$ and
\[
\| \partial_x \Lambda (Y) \|_{j-1} \leq C M \big( 1+\| Y \|_\infty \big)^{j-1} \| \partial_x Y \|_{j-1},
\]
where $C > 0$ is a uniform constant and
\[
M = \sum_{k=1}^j \sup_{\substack{V \in \R^m\\ |V|\leq \| Y \|_\infty}} | D^k_Y \Lambda(Y)| > 0.
\]
\end{lemma}
\begin{proof}
See Vol$'$pert and Hudjaev \cite{VoH72} (see also Lemma 2.4 in \cite{KaTh83}).
\end{proof}

With this information at hand we proceed to estimate the terms $\| N_2 \|_{L^1}$ and $\| N_2 \|_{s}$ inside the integrals in \eqref{estimatef}. First, from \eqref{orderN2} we know that
\[
N_2 = O\big(\rho^2 + m^2 + \rho_x^2 + |\rho||\rho_{xx}|\big),
\]
where $\rho$ and $m$ are the perturbation variables. First, we estimate the term $\rho_x^2$. From Lemma \ref{lemaux1} and by Sobolev imbedding theorem, we get
\[
\Vert \rho_{x}^{2}(\tau) \Vert_{s} \leq 2C\Vert \rho_{x}(\tau)\Vert_{s} \| \rho_{x}(\tau) \|_{L^{\infty}} \leq C \Vert \rho_{x}(\tau) \Vert_{s+1} \Vert \rho(\tau) \Vert_{2},
\]
for any $\tau \in [0,t]$ and because $s \geq 3$.  Applying the Sobolev calculus inequalities from Lemma \ref{lemaux2}, we arrive at
\begin{equation}
\label{N2-s}
\begin{split}
\| N_{2}(\tau) \|_{s} &\leq C \big( \Vert \rho(\tau) \Vert_{s}^{2} + \Vert \rho(\tau) \Vert_{s}\Vert \rho_{xx}(\tau) \Vert_{s} +   \Vert \rho(\tau) \Vert_{2} \Vert \rho_{x}(\tau) \Vert_{s+1} \big)
\\ &\leq C \left( \Vert \rho(\tau) \Vert_{s}^{2}+  \Vert \rho(\tau) \Vert_{s}\Vert \rho_{x}(\tau) \Vert_{s+1}  \right),
\end{split}
\end{equation}
and 
\begin{equation}\label{N2-L1}
\Vert N_{2}(\tau) \Vert_{L^{1}} \leq C \big( \| \rho(\tau)\|_2^2 + \| m(\tau) \|_1^2\big) \leq C \big( \| \rho(\tau)\|_{s}^2 + \| m(\tau) \|_{s-1}^2\big),
\end{equation} 
for all $\tau \in [0,t]$ because $s\geq 3$, where $C > 0$ is a uniform constant. Combine estimates \eqref{estimatef}, \eqref{N2-s} and \eqref{N2-L1} in order to get

\begin{align}
\| \rho(t) \|_s + \| m(t) \|_{s-1} &\leq C (1+t)^{-1/4} \Big( \| \rho_0 \|_s + \|m_0 \|_{s-1}  + \| \rho_0 \|_{L^1} + \|m_0 \|_{L^1}\Big) +  \nonumber\\
&\quad+ C \sup_{0 \leq \tau \leq t} \Vert \rho(\tau) \Vert_{s} \int_{0}^{t}e^{-c_{1}(t-z)}\Vert \rho(\tau) \Vert_{s} \, d\tau \, + \nonumber\\
&\quad+ C \left( \int_{0}^{t}\Vert m_{x}(\tau)\Vert_{s}^{2}\: d\tau \right)^{1/2} \left( \int_{0}^{t}e^{-2c_{1}(t-\tau)}\Vert \rho(\tau) \Vert_{s}^{2}\: d\tau \right)^{1/2}\: + \nonumber\\
&\quad + C \left( \int_{0}^{t}\Vert \rho_{x}(\tau) \Vert_{s+1}^{2}\: d\tau \right)^{1/2} \left( \int_{0}^{t}e^{-2c_{1}(t-\tau)}\Vert \rho(\tau) \Vert_{s}^{2} \: d\tau \right)^{1/2}\: + \nonumber\\
%&\quad + C \int_0^t e^{-c_1(t - \tau)} \| N_2(\tau) \|_s \, d \tau + \label{nonestNs-1.2}\\
&\quad + C \int_0^t (1+t - \tau)^{-3/4} \Big( \| \rho(\tau)\|_{s}^2 + \| m(\tau) \|_{s-1}^2 \Big) \, d \tau. \label{nonestNs-1.2}
\end{align}
%
%\begin{equation}
%\label{nonestHs-1.2}
%\begin{split}
%\Vert U(t) \Vert_{s-1} &\leq C \left(1+t \right)^{-1/4} \left( \Vert U_{0}\Vert_{s-1}+ \Vert U_{0} \Vert_{L^{1}} \right) \:+ \\ & \:\:\:\: + C\sup_{0 \leq z \leq t} \Vert v(z) \Vert_{s} \int_{0}^{t}e^{-c_{1}(t-z)}\Vert v(z) \Vert_{s}\:dz \:+ \\ &\:\:\:\: +C \left( \int_{0}^{t}\Vert u_{x}(z)\Vert_{s}^{2}\: dz \right)^{1/2} \left( \int_{0}^{t}e^{-2c_{1}(t-z)}\Vert v(z) \Vert_{s}^{2}\: dz \right)^{1/2}\: + \\ &\:\:\:\:+ C\left( \int_{0}^{t}\Vert v_{x}(z) \Vert_{s+1}^{2}\: dz \right)^{1/2} \left( \int_{0}^{t}e^{-2c_{1}(t-z)}\Vert v(z) \Vert_{s}^{2} \: dz \right)^{1/2}\: + \\ &\:\:\:\: + C\int_{0}^{t}\left( 1+ t-z\right)^{-3/4} \Vert U(z) \Vert_{s-1}^{2} \: dz.
%\end{split}
%\end{equation}
If we denote
\[
\begin{aligned}
G_{s}(t) &:= \sup_{0 \leq \tau \leq t} \left( 1+ \tau \right)^{1/4} \Big( \| \rho(\tau)\|_{s} + \| m(\tau) \|_{s-1} \Big),\\
Q_s(t) &:= \sup_{0 \leq \tau \leq t} \Big(\Vert \rho(\tau)\Vert_{s+1}^{2}+ \Vert m(\tau) \Vert_{s}^{2}  \Big) + \int_{0}^{t} \left(  \Vert \rho_x(\tau)\Vert_{s+1}^{2} + \Vert m_x(t)\Vert_{s}^{2} \right)\: dt, \\
&= E_s(t) + F_s(t),
\end{aligned}
\]
where $E_s(t)$ and $F_s(t)$ are defined in \eqref{defEs} and \eqref{defFs}, respectively, then we can recast estimate \eqref{nonestNs-1.2} in a simplified form, namely as
\begin{equation}
\label{nonestNs-1.sup}
\begin{aligned}
G_{s}(t) &\leq C \Big( \| \rho_0 \|_s + \|m_0 \|_{s-1}  + \| \rho_0 \|_{L^1} + \|m_0 \|_{L^1}\Big) + C H_{1}(t) Q_s(t)^{1/2} G_{s}(t) +\\
&\qquad  + C H_{2}(t) G_{s}(t)^2,
\end{aligned}
\end{equation}
%\begin{equation}
%\label{nonestNs-1.sup}
%G_{s}(t) \leq C \Big( \| \rho_0 \|_s + \|m_0 \|_{s-1}  + \| \rho_0 \|_{L^1} + \|m_0 \|_{L^1}\Big) + C I_{1}(t)\vertiii{U}_{s,t} E_{s}(t) + C I_{2}(t) E_{s}(t)^2,
%\end{equation}
where 
\begin{equation}\label{mu1}
\begin{split}
H_{1}(t) &:= \sup_{0\leq \tau \leq t}\left( 1+ \tau \right)^{1/4} \int_{0}^{\tau}e^{-c_{1}(\tau-z)}(1+z)^{-1/4}\: dz \: + \\ & \:\:\:\: + \sup_{0\leq \tau \leq t}\left( 1 + \tau \right)^{1/4}\left[ \int_{0}^{\tau}e^{-2c_{1}(\tau-z)}(1+z)^{-1/2} \: dz \right]^{1/2},
\end{split}
\end{equation}
\begin{equation}\label{mu2}
H_{2}(t) := \sup_{0 \leq \tau \leq t}\left( 1+ \tau \right)^{1/4}\int_{0}^{\tau}\left( 1+\tau-z \right)^{-3/4}(1+z)^{-1/2} \: dz.
\end{equation}
Since both integrals, $H_{1}(t)$ and $H_{2}(t)$, are uniformly bounded in $t \geq 0$ (see Lemma A.1 in \cite{PlV22}), we readily obtain the estimate
\begin{equation}
\label{finnonest}
G_{s}(t) \leq C\Big( \| \rho_0 \|_s + \|m_0 \|_{s-1}  + \| \rho_0 \|_{L^1} + \|m_0 \|_{L^1} \!\Big) + C Q_s(t)^{1/2} G_{s}(t)  +  C G_{s}(t)^2.
\end{equation}
Last estimate will be used in a key way to obtain the global decay of perturbations to constant equilibrium states.
%%%

\subsection{Global existence and decay of solutions}

After all these preparations we are ready to prove our main result.

\begin{theorem}[global decay of perturbations of subsonic equilibrium states]
\label{gloexth}
Let $(\rho_*, m_*) \in \R^2$, with $\rho_* > 0$, be a constant equilibrium state of system \eqref{QHD} which satisfies the subsonicity assumption
\[
p'(\rho_*) = \gamma \rho_*^{\gamma -1}> \frac{m_*^2}{\rho_*^2}.
\]
Suppose that $\rho_0  \in H^{s+1}(\R) \cap L^1(\R)$, $m_0  \in H^s(\R)\cap L^1(\R)$, for some $s \geq 3$. Then there exists a positive constant $\e_{2} \leq a_0$ (with $a_{0}$ as in Theorem \ref{themlocale}) such that if 
\begin{equation}
\label{smallcond}
\| \rho_0 \|_{s+1} + \| m_0  \|_{s} + \| \rho_0 \|_{L^1} + \| m_0 \|_{L^1} < \e_2,
\end{equation}
then the Cauchy problem for the QHD system \eqref{QHD} with initial condition $(\rhos+\rho_0,\ms+m_0)(x)$, $x \in \R$, has a unique solution of the form $(\rhos+\rho, \ms+m)(x,t)$ satisfying
\begin{equation}
\label{globsol}
\begin{split}
& \:\: \rho  \in C\left((0,\infty);H^{s+1}(\mathbb{R})) \cap C^{1}((0,\infty); H^{s-1}(\mathbb{R})\right), \\ & \:\: m  \in C\left((0,\infty); H^{s}(\mathbb{R})) \cap C^{1}((0,\infty);H^{s-2}(\mathbb{R})\right) \\ & \:\:(\rho_{x},m_{x})\in L^{2}\left((0,\infty);H^{s+1}(\mathbb{R}) \times H^{s}(\mathbb{R})\right).
\end{split}
\end{equation}
Moreover, the following estimates hold,
%\begin{equation}
%\label{eee3}
%\big(E_{s}(t) + F_{s}(t) \big)^{1/2} \leq C_3 E_s(0),
%\end{equation}
%and
%\begin{equation}
%\label{globdec}
%E_{s-1}(t) \leq C_4 (1+t)^{-1/4} \Big( E_{s-1}(0) + \| \rho_0 -\rhos\|_{L^1} + \| m_0 - \ms \|_{L^1}\Big),
%\end{equation}
%for every $t \geq 0$, some uniform $C_j > 0$ and where
%\begin{align*}
%E_s(t) &= \sup_{\tau \in [0,t]} \Big[ \| \rho(\tau)-\rhos\|^2_{s+1} + \| m(\tau)-\ms\|^2_{s} \Big], 
%\\
%F_s(t) &= \int_0^t \Big[ \| \rho_x(\tau) \|_{s+1}^2 + \| m_x(\tau) \|_s^2 \Big] \, d\tau .
%\end{align*}
%
\begin{equation}\label{globdec}
\| \rho(t)\|_{s} + \| m(t) \|_{s-1} \leq C_{1} \left(1+ t \right)^{-1/4} \!\Big( \| \rho_0 \|_s + \|m_0 \|_{s-1}  + \| \rho_0 \|_{L^1} + \|m_0 \|_{L^1}\Big),
\end{equation}
and
\begin{equation}\label{trinormest}
Q_s(t)^{1/2} \leq C_{2} \Big( \| \rho_0 \|_s + \|m_0 \|_{s-1} \Big),
\end{equation}
for every $t \in [0, \infty)$ and some uniform $C_j > 0$, where
\[
Q_s(t) = \sup_{0 \leq \tau \leq t} \Big(\Vert \rho(\tau)\Vert_{s+1}^{2}+ \Vert m(\tau) \Vert_{s}^{2}  \Big) + \int_{0}^{t} \left(  \Vert \rho_x(\tau)\Vert_{s+1}^{2} + \Vert m_x(t)\Vert_{s}^{2} \right)\: dt.
\]
\end{theorem}
\begin{proof}
%% Change
By virtue of estimate \eqref{finnonest}, we can select $\e_{1}\leq a_{0}$, $\e_{1}$ small enough as in Corollary \ref{cor26}, and $\delta_{1}= \delta_{1}(\e_{1})$ such that for $Q_s(T_1)\leq \e_{1}$ and 
\begin{equation}
\label{vercond}
\| \rho_0 \|_s + \|m_0 \|_{s-1}  + \| \rho_0 \|_{L^1} + \|m_0 \|_{L^1} < \delta_{1},
\end{equation}
there holds
\begin{equation}
\label{nonlest}
\| \rho(t)\|_{s} + \| m(t) \|_{s-1} \leq C_{1}\left( 1+ t\right)^{-1/4} \Big(\| \rho_0 \|_s + \|m_0 \|_{s-1}  + \| \rho_0 \|_{L^1} + \|m_0 \|_{L^1}\Big), 
\end{equation}
for all $t \in [0, T_{1}]$ and some constant $C_{1} = C_{1}(\e_{1}, \delta_{1}) > 1$. Recall that the local solution to the initial value problem, belonging to $X_{s}(0, T_{1}; m_{0}/2, 2M_{0})$, for some $T_{1}= T_{1}(a_{0})$, exists for all $t \in [0,T_1]$ thanks to Theorem \ref{themlocale}. Next, we define
\[
 \e_{2} := \mbox{min} \left\lbrace \e_1,\frac{\e_1}{C_{0}} , \frac{\e_1}{C_{2}\left( 1+C_{0}^{2}\right)^{1/2}}, \delta_{1} \right\rbrace > 0.
 \]
Let us suppose that condition \eqref{smallcond} holds for this selected value of $\e_2$. Whence, the local existence theorem \ref{themlocale} implies that
\[
Q_s(T_1) = E_s(T_1) + F_s(T_1) \leq C_1 E_s(0) = C_0 \big( \| \rho_0 \|_{s+1} + \| m_0 \|_s \big) < C_0 \e_2 \leq \e_1.
\]
This bound, together with
\[
\| \rho_0 \|_s + \|m_0 \|_{s-1}  + \| \rho_0 \|_{L^1} + \|m_0 \|_{L^1} < \e_2 \leq \delta_{1},
\]
readily implies estimate \eqref{nonlest} for $t \in [0, T_{1}]$. In addition, we have
\[
E_s(T_1) \leq Q_s(T_1) \leq \e_1.
% \sup_{0 \leq t \leq T_{1}}  \Vert  U (t)\Vert_{s} \leq \vertiii{U}_{s,T_{1}} \leq \e_1. 
\] 
Hence, we have verified condition \eqref{localaprioriEE} from Corollary \ref{cor26}. Upon application of Corollary \ref{cor26} we obtain
\[
Q_s(T_1)^{1/2} \leq C_2 E_s(0)^{1/2} = C_2 \big( \| \rho_0 \|_{s+1} + \| m_0 \|_s \big).
\]
By virtue of 
\[
\| \rho(T_1) \|_{s+1} + \| m(T_1) \|_{s} \leq Q_s(T_1)^{1/2} \leq \e_1,
\]
we can consider the Cauchy problem with initial condition at $t = T_1$ in order to find a local solution in $[T_{1}, 2T_{1}]$ satisfying 
%estimate 
\[
\begin{aligned}
\sup_{T_1 \leq \tau \leq 2T_1} \Big(\Vert \rho(\tau)\Vert_{s+1}^{2}+ \Vert m(\tau) \Vert_{s}^{2}  \Big) &+ \int_{T_1}^{2T_1} \left(  \Vert \rho_x(\tau)\Vert_{s+1}^{2} + \Vert m_x(t)\Vert_{s}^{2} \right)\: dt \leq \\
&\leq C_0^2 \big( \| \rho(T_1) \|_{s+1} + \| m(T_1) \|_s \big)^2\\
&\leq C_0^2 Q_s(T_1)^2.
\end{aligned}
\]
Therefore, we obtain
\[
\begin{aligned}
Q_s(2T_1)^{1/2} &= \left[ Q_s(T_1) + \!\!\sup_{T_1 \leq \tau \leq 2T_1} \Big(\Vert \rho(\tau)\Vert_{s+1}^{2}+ \Vert m(\tau) \Vert_{s}^{2}  \Big) + \int_{T_1}^{2T_1} \left(  \Vert \rho_x(\tau)\Vert_{s+1}^{2} + \Vert m_x(t)\Vert_{s}^{2} \right)\: dt \right]^{1/2} \\
&\leq (1 + C_0^2)^{1/2} Q_s(T_1)^{1/2} \\
&\leq C_2 (1+C_0^2)^{1/2} \big( \| \rho_0 \|_{s+1} + \| m_0 \|_s \big)\\
&\leq C_2 (1+C_0^2)^{1/2} \big( \| \rho_0 \|_{s+1} + \| m_0 \|_s + \| \rho_0 \|_{L^1} + \|m_0 \|_{L^1}\big)\\
&< C_2 (1+C_0^2)^{1/2} \e_2 \\
&\leq \e_1.
\end{aligned}
\]
This yields,
\[
E_s(2T_1)^{1/2} \leq Q_s(2T_1)^{1/2} < \e_1.
\]
This estimate, together with the already verified condition \eqref{vercond} allows us to obtain estimate \eqref{trinormest} and the condition \eqref{localaprioriEE} from Corollary \ref{cor26}, but now on the time interval $t \in [0, 2T_1]$. Consequently, 
\[
\| \rho(t)\|_{s} + \| m(t) \|_{s-1} \leq C_{1}\left( 1+ t\right)^{-1/4} \Big(\| \rho_0 \|_s + \|m_0 \|_{s-1}  + \| \rho_0 \|_{L^1} + \|m_0 \|_{L^1}\Big),
\]
holds for all $t \in [0,2T_1]$ and 
\[
Q_s(2T_1)^{1/2} \leq C_2 \big( \| \rho_0 \|_{s+1} + \| m_0 \|_s \big).
\]
We can proceed by iteration in order to obtain estimates \eqref{trinormest} and \eqref{globdec} for the time intervals $[0, 3T_{1}]$, $[0,4T_1]$, and so on. Thus, the estimates hold globally in time. The theorem is now proved.
\end{proof}

\section*{Acknowledgements}
The work of D. Zhelyazov was supported by a Post-doctoral Fellowship by the Direcci\'{o}n General de Asuntos del Personal Acad\'{e}mico (DGAPA), UNAM. The work of R. G. Plaza was partially supported by DGAPA-UNAM, program PAPIIT, grant IN-104922.

\appendix
\section{Proof of Lemma \ref{lemsupersonic}}
\label{secappen}
Under the supersonicity assumption \eqref{supersonic}, let us define the positive constant
\[
\beta_* := \frac{m_*^2}{\rho_*^2} - p'(\rho_*) > 0.
\]
Assume $\hU = \hU(\xi)$ is a non-trivial solution to \eqref{spect}. Therefore $\hU \in \ker D(\lambda,\xi)$, where
\[
D(\lambda, \xi) = \lambda I + i \xi A_* + \xi^2 B_* + i \xi^3 C_* = \begin{pmatrix} \lambda & i \xi \\ - i \xi \beta_* + \tfrac{1}{2} i \xi^3 k^2 & \lambda + i 2 \xi m_*/\rho_* + \mu \xi^2 \end{pmatrix}.
\]
This yields the dispersion relation 
\begin{equation}
\label{disprel}
\det D(\lambda, \xi) = \lambda^2 + \Big( \mu \xi^2 + i \frac{2 \xi m_*}{\rho_*}\Big) \lambda + \xi^2( \tfrac{1}{2} \xi^2 k^2 - \beta_* \big)= 0.
\end{equation}
The discriminant of the second order polynomial in $\lambda$ on the left hand side of \eqref{disprel} is $\Delta(\xi) := a(\xi) + i b(\xi)$, where
\[
a(\xi) := \xi^2 \big( \xi^2(\mu^2 - 2 k^2) - 4 p'(\rho_*) \big),\qquad b(\xi) := 4 \mu \xi^3 \frac{m_*}{\rho_*}.
\]
Therefore, the roots of \eqref{disprel} are $\lambda_\pm(\xi) := - \tfrac{1}{2} \mu \xi^2 - i \xi m_*/\rho_* \pm \tfrac{1}{2} \Delta(\xi)^{1/2}$. Let us examine
\[
\Re \lambda_+(\xi) = - \tfrac{1}{2} \mu \xi^2 + \tfrac{1}{2} \Re \Delta(\xi)^{1/2}.
\]
We now show that $\Re \lambda_+(\xi) > 0$ for $0 < |\xi| \ll1$, sufficiently small. This is equivalent to prove that
\begin{equation}
\label{good}
\Re \Delta(\xi)^{1/2} > \mu \xi^2, \qquad \text{for } \; 0 < |\xi| \ll1.
\end{equation}
Recalling that
\[
\Re \Delta(\xi)^{1/2} = \frac{1}{\sqrt{2}} \sqrt{a(\xi) + \sqrt{a(\xi)^2 + b(\xi)^2}},
\]
we observe that \eqref{good} is equivalent to
\[
a(\xi) + \sqrt{a(\xi)^2 + b(\xi)^2} > 2 \mu^2 \xi^4,
\]
for $|\xi| \approx 0^+$. Since,
\[
\begin{aligned}
2 \mu^2 \xi^4 - a(\xi) &= 2 \mu^2 \xi^4 - \xi^2 \big( \xi^2(\mu^2 - 2 k^2) - 4 p'(\rho_*) \big)\\
&= \xi^2 \big( (\mu^2  + 2 k^2) \xi^2 + 4 p'(\rho_*) \big) > 0,
\end{aligned}
\]
then \eqref{good} holds if and only if
\[
b(\xi)^2 > 4 \mu^4 \xi^8 - 4 a(\xi) \mu^2 \xi^4.
\]
Upon substitution of the expressions for $a(\xi)$ and $b(\xi)$ we reckon that \eqref{good} is satisfied if and only if
\[
\xi^2 \Big( 2 \frac{m_*^2}{\rho_*^2} - 2 p'(\rho_*) - k^2 \xi^2 \Big) = \xi^2 (2 \beta_* + O(\xi^2)) > 0,
\]
as $|\xi| \to 0$. But this is true because of the supersonicity condition ($\beta_* > 0$). We conclude that $\Re \lambda_+(\xi) > 0$ for sufficiently small values of $|\xi|$. The lemma is proved.
\qed

\def\cprime{$'$} \def\cprime{$'\!\!$}\def\cprime{$'\!\!$}\def\cprime{$'\!\!$}

\end{document}